\documentclass{tac}

\usepackage{amsmath}
\usepackage{amssymb}
\usepackage{amscd}
\usepackage{tikz}
\usetikzlibrary{decorations.markings}
\usepackage[all]{xy}
\xyoption{2cell}
\SelectTips{cm}{}
\usepackage{enumitem}
\usepackage{pgfplots}
\pgfplotsset{compat=1.16}
\usepackage{arydshln}

\newcommand{\C}{\mathbb{C}}
\newcommand{\M}{\mathsf{matr}}
\newcommand{\pb}[1][dr]{\save*!/#1-1.5pc/#1:(-1,1)@^{|-}\restore}
\renewcommand{\mod}{\mathsf{mod}\,}

\newdir{ >}{{}*!/-8pt/@{>}}

\title{The matrix taxonomy of\\ finitely complete categories}

\author{Michael Hoefnagel, Pierre-Alain Jacqmin, and Zurab Janelidze}

\thanks{The second author would like to warmly thank Stellenbosch University for its kind hospitality during his visit in January 2020, when part of the project was realised. He is also grateful to the FNRS for its generous support. The third author acknowledges financial support of the South African National Research Foundation. The authors are also grateful to the anonymous referee for valuable remarks made to improve the readability of the paper.}

\address{M.H.:\\
Mathematics Division \\ Department of Mathematical Sciences \\ Stellenbosch University \\ Private Bag X1 Matieland 7602 \\ South Africa \\ \textrm{Email: mhoefnagel@sun.ac.za} \\ \\
National Institute for Theoretical and Computational Sciences (NITheCS) \\ South Africa\\ \\
P.-A.J.:\\
Institut de Recherche en Math\'ematique et Physique \\ Universit\'e catholique de Louvain \\ Chemin du Cyclotron 2 \\ B 1348 Louvain-la-Neuve \\ Belgium \\ \textrm{Email: pierre-alain.jacqmin@uclouvain.be} \\ \\
Centre for Experimental Mathematics \\ Department of Mathematical Sciences \\ Stellenbosch University \\ Private Bag X1 Matieland 7602 \\ South Africa \\ \\
Z.J.:\\
Mathematics Division \\ Department of Mathematical Sciences \\ Stellenbosch University \\ Private Bag X1 Matieland 7602 \\ South Africa \\ \textrm{Email: zurab@sun.ac.za}  \\ \\
National Institute for Theoretical and Computational Sciences (NITheCS) \\ South Africa
}


\copyrightyear{2022}

\begin{document}

\maketitle
\begin{abstract}
This paper is concerned with the taxonomy of finitely complete categories, based on `matrix properties' --- these are a particular type of exactness properties that can be represented by integer matrices. In particular, the main result of the paper gives an algorithm for deciding whether a conjunction of such properties implies another such property. Computer implementation of this algorithm allows one to peer into the complex structure of the poset of `matrix classes', i.e., the poset of all collections of finitely complete categories determined by matrix properties. Among elements of this poset are the collections of Mal'tsev categories, majority categories, (finitely complete) arithmetical categories, as well as finitely complete extensions of various classes of varieties defined by a special type of Mal'tsev conditions found in the literature.
\end{abstract}

\section*{Introduction}

This paper deals with so-called matrix properties of categories, introduced and initially studied in \cite{JanelidzeMSc,JanelidzePhD,ZJanelidze2006a,ZJanelidze2006b}. The problem of finding an algorithm for deciding when does one matrix property imply another has been unsolved since then. In the present paper we solve this problem.

In what follows, we will work with matrices of (non-negative) integers, such as
$$\left[\begin{array}{ccccc} 1 & 1 & 0 & 2 & 2 \\ 0 & 0 & 1 & 1 & 0 \\ 2 & 0 & 1 & 2 & 1 \end{array}\right].$$
Such matrices encode systems of `linear equations' in the algebraic theory of a variety of universal algebras, which in the specific example above is:
$$\left\{\begin{array}{ccc} p(x_1, x_1, x_0, x_2, x_2)=x_0, \\ p(x_0, x_0, x_1, x_1, x_0)=x_0, \\ p(x_2, x_0, x_1, x_2, x_1)=x_0. \end{array}\right.$$
Recall that a variety of universal algebras is the collection of all algebraic structures, defined by some signature of operators satisfying certain identities, such as the ones above, where the $x_i$'s are variables (so they may assume arbitrary values in a specific algebra) and $p$ is one of the operators. The system of equations above gives a property of a variety stating that using the operators in its signature one may build an operator $p$ such that every algebra in the variety will model the system (every equation in the system will hold for all values of the $x_i$'s in the algebra) --- for a variety to have this property is, thus, a particular type of Mal'tsev condition~\cite{Taylor1973}. For example, in the variety of vector spaces over a two-element field, we can define 
$$p(x_1,x_2,x_3,x_4,x_5)=x_1+x_2+x_3+x_4+x_5$$
to satisfy the given identities. This means that the variety in question has the property determined by the matrix. Notice that the matrix is obtained from such a system of linear equations as the matrix of indices of variables in the system, appearing to the left of the equality sign. The variables on the right of the equality sign by default all have index $0$ (otherwise, we must `extend' the matrix with another column).

Algebras in a variety form a category, where morphisms are homomorphisms of algebras. The property of a variety that the system of equations determines can be reformulated as a property of this category. To do this, one first wants to see the matrix as a property of an $n$-ary relation, where $n$ is the number of rows of the matrix (so $n=3$ in the example above). This is done by means of `column-vectors' of the matrix --- in the case of the matrix above, the corresponding property of a ternary relation $R$ states:
$$\left\{\left[\begin{array}{c} x_1 \\ x_0 \\ x_2 \end{array}\right],\left[\begin{array}{c} x_1 \\ x_0 \\ x_0 \end{array}\right], \left[\begin{array}{c} x_0 \\ x_1 \\ x_1 \end{array}\right], \left[\begin{array}{c} x_2 \\ x_1 \\ x_2 \end{array}\right], \left[\begin{array}{c} x_2 \\ x_0 \\ x_1 \end{array}\right]\right\}\subseteq  R\quad\Longrightarrow\quad \left[\begin{array}{c} x_0 \\ x_0 \\ x_0 \end{array}\right]\in R.$$
The property of a category determined by the matrix states that all \emph{internal} $n$-ary relations in the category must have the property above \emph{internalised} to the category (this process of `internalisation' is the standard one given by the Yoneda embedding). We consider this property together with the requirement that the category is finitely complete. The Mal'tsev condition on a variety obtained from a matrix is equivalent to the corresponding condition on internal relations in the category of algebras (categories of algebras are always finitely complete). This is an easy theorem to prove, using standard universal-algebraic techniques (see~\cite{ZJanelidze2006a}), which can be illustrated in the case of vector spaces and the matrix above by saying that the existence of a linear combination $p$ that fulfills the above system of equations is equivalent to every subspace $R$ of $V^n$ (for every vector space~$V$) to satisfy the implication above. Thus, collections of varieties given by Mal'tsev conditions such as the one above extend to collections of finitely complete categories given by matrices. We refer to these collections of finitely complete categories as `matrix classes'.

Among examples of matrix classes are the matrix class of Mal'tsev categories in the sense of~\cite{CarboniPedicchioPirovano1991,CarboniLambekPedicchio1991} and the matrix class of majority categories in the sense of~\cite{Hoefnagel2019a,Hoefnagel2020a}. Their universal-algebraic counterparts are given by the collections of Mal'tsev varieties~\cite{Maltsev1954,Smith1976} and varieties admitting a majority term~\cite{Pixley1963}. The intersection of these two collections is given by arithmetical varieties~\cite{Pixley1963,Pixley1971,Pixley1972}. Arithmetical categories in the sense of~\cite{Pedicchio1996} are Barr-exact categories with coequalisers in the matrix class given by a matrix that determines the Mal'tsev condition for arithmetical varieties. Many more matrix classes can be produced from various Mal'tsev conditions commonly encountered in universal algebra. This brings us to the following questions: how many matrix classes are there, for given dimensions of a matrix, and how can we decide whether two matrices give rise to the same matrix class or not?

In this paper we present an algorithm for determining implications of finite conjunctions of matrix properties, which leads to answering the questions stated above, with the help of a computer, for matrices with sufficiently small dimensions. For example, it allows one to generate Figure~\ref{fig:4x5x2}, which gives a computation of the poset of all `non-degenerate' matrix classes given by matrices in $\M(4,5,2)$ ordered by inclusion, where $\M(4,5,2)$ denotes the set of matrices having $4$ rows, $5$ columns and whose each entry is either $0$ or~$1$ (we will explain what `non-degenerate' means shortly). Each of the matrix classes in Figure~\ref{fig:4x5x2} is represented by a `canonical matrix' (defined in Section~\ref{sec4}) which we display as a grid, where each white square represents the entry $0$ and each grey square represents the entry~$1$. The reason why some of the matrices in Figure~\ref{fig:4x5x2} appear to have smaller dimensions than $4\times 5$ is that we can always add duplicate rows and columns to a matrix without changing the matrix class (as confirmed by Proposition~1.7 in~\cite{ZJanelidze2006b}). By a `non-degenerate' matrix class we mean one that is defined by a non-empty matrix and is not equal to the collection of finitely complete preorders or the collection of all finitely complete categories (see Section~\ref{sec2} for a discussion of the empty matrix). The algorithm presented in this paper is first of its kind in the broader field of category theory concerned with general methods in the study of exactness properties of categories.

\begin{figure}[htb]
    \centering
    \includegraphics[width=400pt]{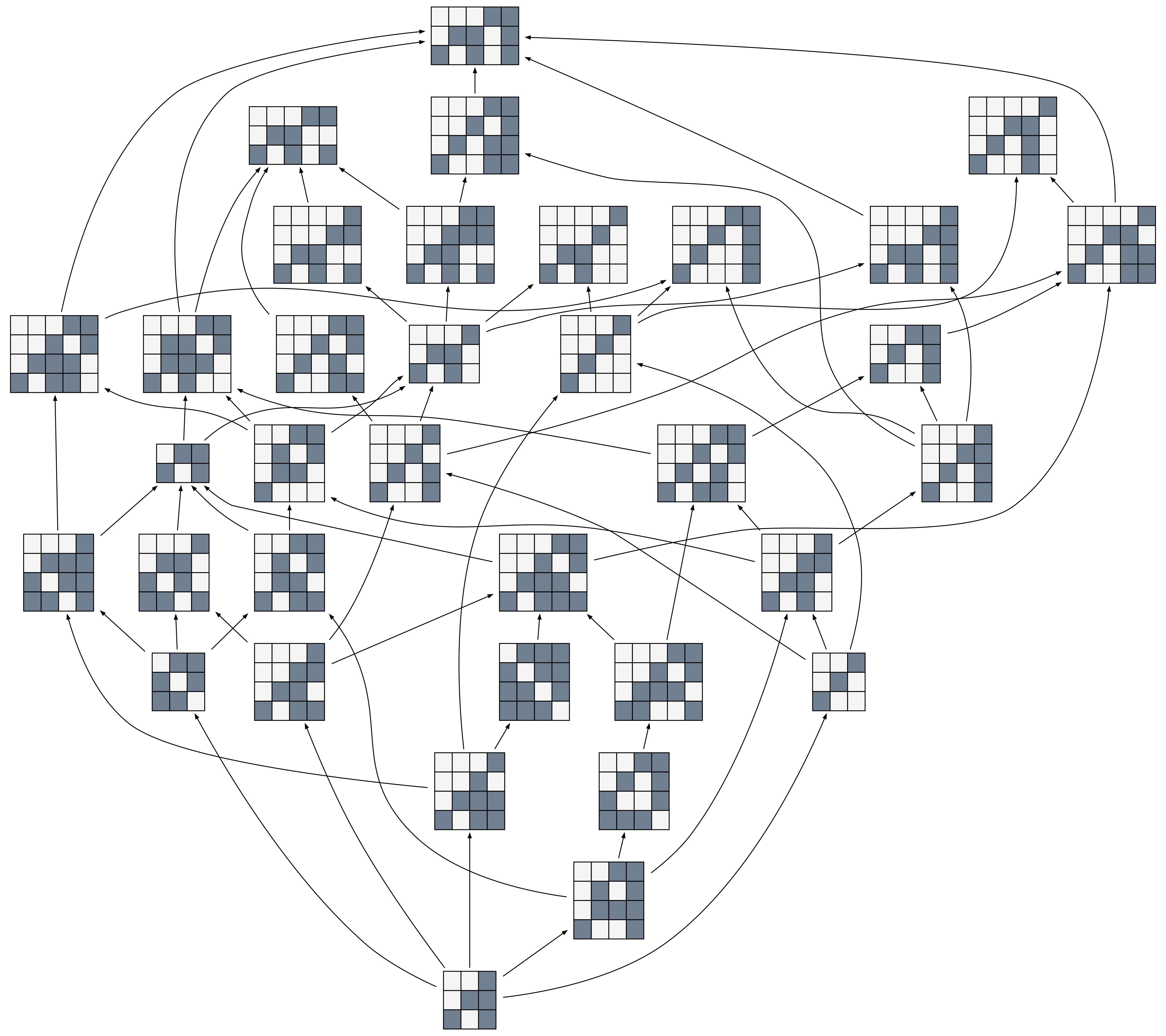}
    \caption{Hasse diagram of the poset of non-degenerate matrix classes in $\mathsf{Mclex}[4,5,2]$.}
    \label{fig:4x5x2}
\end{figure}

Let us make a similar remark as in~\cite{CarboniKellyPedicchio1993}, one of the pioneering papers on Mal'tsev categories: \textit{such a pursuit would have no value unless Mal'tsev categories were of fairly common occurrence in mathematics}; except,  in our case, `Mal'tsev categories' should be replaced with `categories having matrix properties'. Firstly, there are the algebraic examples. For instance, in the Mal'tsev case, these are the categories of groups, rings, modules, Lie algebras, loops, Heyting algebras, and many others. Then, such structures internal to a given finitely complete category provide examples as well. Still in the Mal'tsev case, this gives us, say, the category of topological groups as an example. Examples of a different kind are given by duals of geometric categories. For instance, the dual of the category of topological spaces is a majority category (but not a Mal'tsev category). 
See~\cite{Hoefnagel2019a,Weighill2017} for some other examples. It is worth remarking that the dual of the category of sets belongs to all non-degenerate matrix classes. Also, matrix classes are closed under reflective subcategories and various natural categorical constructions, such as taking categories of functors from any fixed category to a category in the matrix class.

The results of this present paper, which actually build on ideas contained already in \cite{HoefnagelPhD,Hoefnagel2019a,Hoefnagel2020a} and~\cite{Weighill2017}, reveal a new realm of geometric examples of members of matrix classes that are given by subcategories of duals of categories of relations. It is these examples that lead us to the algorithm presented herein for deciding implications of matrix properties. We must point out that, as already mentioned in the remark following Corollary~5.13 in~\cite{HoefnagelPhD}, these same examples show that there are implications of matrix properties that hold in the algebraic case, but not in the general finitely complete case. In other words, implications of matrix properties are \emph{context sensitive} (see Section~\ref{sec:context sensitivity}).

The paper is organised as follows. Apart from this Introduction, there are five sections. Section~\ref{sec1} is in most part recollection of necessary concepts from the literature dealing with matrix properties. Section~\ref{sec2} develops these concepts further along with distinguishing and studying a special class of matrices, called `trivial matrices', which require a separate treatment in proofs of the results contained in subsequent sections. Section~\ref{sec3} contains the main results of the paper. These results lead to the algorithm for deciding implications of (conjunctions of) matrix properties, which is formulated in the same section. Section~\ref{sec4} gives computation of some fragments of the poset of matrix classes, based on a computer implementation of the algorithm. We conclude with Section~\ref{sec:context sensitivity}, which discusses context-sensitivity of matrix properties.

\renewcommand\contentsname{\vspace*{-1cm}}
\tableofcontents

\section{Matrix classes}\label{sec1}

We assume basic knowledge of categories such as (finite) limits, colimits and functors, such as can be found in~\cite{MacLane1998}. For a category~$\mathbb{C}$, it is customary to denote the set of all morphisms from an object $X$ to an object $Y$ in $\mathbb{C}$ by $\mathbb{C}(X,Y)$.

Given integers $n,k > 0$ and $m \geqslant 0$, we write $\M(n,m,k)$ for the set of all $n\times m$ matrices whose entries are elements of the set $\{0,\dots,k-1\}$. Given a sequence $(S_1,\dots,S_n)$ of sets, a \emph{row-wise interpretation} of $A\in\M(n,m,k)$ of type $(S_1,\dots,S_n)$, is an $n\times m$ matrix $B$ whose entries are given by $b_{ij}=f_i(a_{ij})$, where each $f_i$ is a specified function $f_i\colon \{0,\dots,k-1\}\to S_i$ and the $a_{ij}$'s are the entries of~$A$. We will often display a row-wise interpretation of $A$ as a matrix extended with the column of values~$f_i(0)$:
$$\left[\begin{array}{ccc|c} f_1(a_{11}) & \dots & f_1(a_{1m}) & f_1(0)\\ \vdots & & \vdots & \vdots \\ f_n(a_{n1}) & \dots & f_n(a_{nm}) &  f_n(0) \end{array}\right]$$
We write $[0]_n$ for a column-matrix with $n$ rows, all of whose entries are~$0$. Next, we present in this language the definition of the well-known concept of an `internal relation' in a category.

An \emph{internal $n$-ary relation} (between objects $C_1,\dots,C_n$) in a category $\mathbb{C}$ is given by an object $R$ together with a row-wise interpretation 
$$r=\left[\begin{array}{c} r_1 \\ \vdots \\ r_n \end{array}\right]$$
of $[0]_n\in\M(n,1,1)$ of type $(\mathbb{C}(R,C_1),\dots,\mathbb{C}(R,C_n))$ (so, it is given by a pair $(R,r)$) such that
$$\left[\begin{array}{c} r_1x \\ \vdots \\ r_nx \end{array}\right]=\left[\begin{array}{c} r_1y \\ \vdots \\ r_ny \end{array}\right]\quad\Longrightarrow\quad x=y$$
for any two parallel morphisms $x,y\colon X\to R$ with an arbitrary domain~$X$. We sometimes denote such a relation $(R,r)$ simply by~$r$. When $\mathbb{C}$ has finite products, an internal $n$-ary relation in $\mathbb{C}$ can also be viewed as a monomorphism $r\colon R\rightarrowtail C_1\times\dots\times C_n$, with $r_i=\pi_ir$, where $\pi_i$ denotes $i$-th product projection $\pi_i\colon C_1\times\dots\times C_n\to C_i$.

We remind the reader that up to conceptual identification of monomorphisms into an object with `subobjects' of that object, we have the following:
\begin{itemize}
    \item Internal $n$-ary relations in the category $\mathbf{Set}$ of sets are the same as the usual $n$-ary relations between sets. For this reason, definitions involving internal relations in a category can be particularised to relations between sets.
    
    \item More generally, internal $n$-ary relations in an algebraic category (by which, in this paper, we mean the category of algebras of a given single-sorted variety of universal algebras) are the same as $n$-ary homomorphic relations (recall that a homomorphic relation between, say, groups $C_1,\dots,C_n$ is a subgroup of the cartesian product $C_1\times\dots\times C_n$). Thus, for instance, a congruence on a group is an example of an internal $2$-ary (binary) relation in the category of groups.
\end{itemize}

For an internal $n$-ary relation $(R,r)$ between objects $C_1,\dots,C_n$ in a category~$\mathbb{C}$, a matrix $M\in\M(n,m,k)$ and an object~$X$, we say that $r$ is \emph{compatible} with a row-wise interpretation 
$$\left[\begin{array}{ccc|c} x_{11} & \dots & x_{1m} & y_1\\ \vdots & & \vdots & \vdots\\ x_{n1} & \dots & x_{nm} & y_n \end{array}\right]$$
of $M$ of type $(\mathbb{C}(X,C_1),\dots,\mathbb{C}(X,C_n))$ when, if there exist morphisms $u_1,\dots,u_m\colon X\to R$ such that 
$$\left[\begin{array}{c} x_{1j} \\ \vdots \\ x_{nj} \end{array}\right]=\left[\begin{array}{c} r_1u_j \\ \vdots \\ r_nu_j \end{array}\right]$$
for each $j\in\{1,\dots,m\}$, then there exists a morphism $v\colon X\to R$ such that
$$\left[\begin{array}{c} y_{1} \\ \vdots \\ y_{n} \end{array}\right]=\left[\begin{array}{c} r_1v \\ \vdots \\ r_nv \end{array}\right].$$
We say that $r$ is \emph{strictly $M$-closed over $X$} when it is compatible with every row-wise interpretation of $M$ of type $(\mathbb{C}(X,C_1),\dots,\mathbb{C}(X,C_n))$.

\begin{example}\label{ExaD}
Consider the category $\mathbf{Gp}$ of groups and let $X$ be any group. We will now demonstrate how every homomorphic relation $R$ between groups $C$ and $C'$ is strictly $M$-closed over~$X$, where
$$M=\left[\begin{array}{ccccc} 0 & 1 & 1 & 2 & 2\\ 1 & 1 & 2 & 2 & 0 \end{array}\right].$$
Firstly, we remark that the relation $R$ is seen as an internal binary relation $(R,r)$, with $r_1\colon R\to C$ and $r_2\colon R\to C'$ given by $r_1(a,b)=a$ and $r_2(a,b)=b$. 
A row-wise interpretation of $M$ of type $(\mathbb{C}(X,C),\mathbb{C}(X,C'))$ is given by a matrix
$$\left[\begin{array}{ccccc|c} x_0 & x_1 & x_1 & x_2 & x_2 & x_0\\ x'_1 & x'_1 & x'_2 & x'_2 & x'_0 & x'_0 \end{array}\right],$$
where $x_i\colon X\to C$ and $x'_i\colon X\to C'$ are arbitrary group homomorphisms, for $i\in\{0,1,2\}$. We must show that if there exist group homomorphisms $u_1,\dots,u_5\colon X\to R$ such that for all $w\in X$ we have
\begin{align*}
u_1(w)&=(x_0(w),x'_1(w)),\\
u_2(w)&=(x_1(w),x'_1(w)),\\
u_3(w)&=(x_1(w),x'_2(w)),\\
u_4(w)&=(x_2(w),x'_2(w)),\\
u_5(w)&=(x_2(w),x'_0(w)),
\end{align*}
then there exists a group homomorphism $v\colon X\to R$ such that $v(w)=(x_0(w),x'_0(w))$ for all $w\in X$. Note that since $x_0$ and $x'_0$ are homomorphisms, we just need to show that a function $v\colon X\to R$ exists satisfying the equality above for all $w\in X$. Define $v$ as follows:
$$v(w)=u_1(w)-u_2(w)+u_3(w)-u_4(w)+u_5(w).$$
Then 
\begin{align*}
&v(w)\\
&=u_1(w)-u_2(w)+u_3(w)-u_4(w)+u_5(w)\\
&= (x_0(w)-x_1(w)+x_1(w)-x_2(w)+x_2(w),x'_1(w)-x'_1(w)+x'_2(w)-x'_2(w)+x'_0(w))\\
&=(x_0(w),x'_0(w))
\end{align*}
as desired.
\end{example}

Given a matrix $M\in\M(n,m,k)$ and a set~$S$, an \emph{interpretation} (without `row-wise') of $M$ of type $S$ is a row-wise interpretation of $M$ of type $(S,\dots,S)$ for which $f_1=\dots=f_n$.
An internal $n$-ary relation between identical objects $C,\dots,C$ in~$\mathbb{C}$ (which we will, in the future, refer to as an internal relation `on'~$C$) is \emph{$M$-closed} (without `strictly') over an object $X$ when it is compatible with any interpretation of $M$ of type $\mathbb{C}(X,C)$. Following~\cite{ZJanelidze2006a}, we say that
\begin{itemize}
    \item an internal relation $r$ (between objects $C_1,\dots,C_n$) is \emph{strictly $M$-closed}, when $r$ is strictly $M$-closed over every object $X$ in the category;
    
    \item an internal relation $r$ on an object $C$ is \emph{$M$-closed}, when $r$ is $M$-closed over every object $X$ in the category;
    
    \item the given category $\mathbb{C}$ \emph{has $M$-closed relations} when every internal $n$-ary relation on any object in $\mathbb{C}$ is $M$-closed; when $\mathbb{C}$ is finitely complete, this is equivalent to every internal $n$-ary relation in $\mathbb{C}$ being strictly $M$-closed (by Theorem~2.4 in~\cite{ZJanelidze2006a}).
\end{itemize}

Strict $M$-closedness of a homomorphic relation between algebras $A_1,\dots,A_n$, seen as an internal relation in some algebraic category where $A_1,\dots,A_n$ are objects, is equivalent to its strict $M$-closedness over the free algebra with one generator, which itself is equivalent to the strict $M$-closedness of the same relation seen as an internal relation in the category of sets. The latter is nothing but the statement that for any row-wise interpretation
$$\left[\begin{array}{ccc|c} x_{11} & \dots & x_{1m} & y_1\\ \vdots & & \vdots & \vdots\\ x_{n1} & \dots & x_{nm} & y_n \end{array}\right]$$
of $M$ of type $(A_1,\dots,A_n)$, we have:
$$\left\{\left[\begin{array}{c} x_{11} \\ \vdots \\ x_{n1} \end{array}\right],\dots, \left[\begin{array}{c} x_{1m} \\ \vdots \\ x_{nm} \end{array}\right]\right\}\subseteq  R\quad\Longrightarrow\quad \left[\begin{array}{c} y_1 \\ \vdots \\ y_n \end{array}\right]\in R.$$
In the context of Example~\ref{ExaD} this means that to check that $R$ was strictly $M$-closed over arbitrary~$X$, for the matrix $M$ given there, all we needed to do is to check that we always have 
$$\left\{\left[\begin{array}{c} x_0 \\ x'_1 \end{array}\right], \left[\begin{array}{c} x_1 \\ x'_1 \end{array}\right],\left[\begin{array}{c} x_1 \\ x'_2 \end{array}\right],\left[\begin{array}{c} x_2 \\ x'_2 \end{array}\right],\left[\begin{array}{c} x_2 \\ x'_0 \end{array}\right]\right\}\subseteq  R\quad\Longrightarrow\quad \left[\begin{array}{c} x_0 \\ x'_0 \end{array}\right]\in R,$$
where this time $x_i\in C$ and $x'_i\in C'$ for $i\in\{0,1,2\}$. We have this indeed thanks to the equality
$$\left[\begin{array}{c} x_0 \\ x'_1 \end{array}\right]- \left[\begin{array}{c} x_1 \\ x'_1 \end{array}\right]+\left[\begin{array}{c} x_1 \\ x'_2 \end{array}\right]-\left[\begin{array}{c} x_2 \\ x'_2 \end{array}\right]+\left[\begin{array}{c} x_2 \\ x'_0 \end{array}\right]= \left[\begin{array}{c} x_0 \\ x'_0 \end{array}\right].$$
Similar remarks apply to the notion of $M$-closedness.

\begin{example}
A transitive relation $R$ on a set $C$ in $\mathbf{Set}$ is always $M$-closed, when
$$M=\left[\begin{array}{cc} 0 & 1 \\ 1 & 0 \end{array}\right],$$
since if $(x_0,x_1)\in R$ and $(x_1,x_0)\in R$, then $(x_0,x_0)\in R$. However, $R$ is not necessarily strictly $M$-closed, since that would mean that $(x_0,x'_0)\in R$ every time $(x_0,x'_1)\in R$ and $(x_1,x'_0)\in R$. For instance, if $R$ was reflexive, strict $M$-closedness would force $R=C\times C$. 
\end{example}

We must point out that the notation for matrices $M$ used in this paper is different from the usual notation originally introduced in~\cite{JanelidzeMSc,JanelidzePhD,ZJanelidze2006a}, where the study of categories with $M$-closed relations, for a general~$M$, takes its start. In the usual notation, matrices are filled in with variables instead of integers, and contain an additional right column, just like the display of a row-wise interpretation above. In that case, if we consider pointed categories, one usually uses the $0$'s to represent a zero morphism (or a constant in the algebraic interpretation). This is not a problem for the present paper, since here we do not consider the pointed context; our $0$'s will thus always represent a default variable used to fill the right column. We must note here that some of the results of this paper have analogues in the pointed case, and that some do not. For a separate treatment of the pointed case we refer the reader to~\cite{HoefnagelJacqmin2021}.

For a given~$M$, the collection of all finitely complete categories with $M$-closed relations is denoted by $\mathsf{mclex}\{M\}$ and is called a \emph{matrix class} (of finitely complete categories) in this paper (this notation abbreviates the term `matrix class of left exact categories' in which `left exact category' is an alternative name for a `finitely complete category'). By a \emph{matrix property} we mean the property of a finitely complete category to have $M$-closed relations for a given matrix~$M$ (as in~\cite{JanelidzePhD}).

\begin{remark}
Note that matrix classes are `too large' to be classes in the usual set-theoretic sense. For this reason, without the term `matrix' we refer to them as `collections'. Another possible name for `matrix class' is `matrix family', which would mimic the name `Mal'tsev family' used in universal algebra, as introduced in~\cite{FreeseMcKenzie2017}. 
\end{remark}

\begin{example}\label{ExaA}
The most standard example of a collection of finitely complete categories with $M$-closed relations is the one where
$$M=\left[\begin{array}{ccc} 0 & 1 & 1\\ 1 & 1 & 0 \end{array}\right].$$
An internal binary relation $r\colon R\rightarrowtail C_1 \times C_2$ in a finitely complete category is strictly $M$-closed with respect to this matrix if and only if, for each object $X$ and each morphisms $f_0,f_1\colon X\to C_1$ and $g_0,g_1\colon X\to C_2$, if the three induced morphisms
$$\left[\begin{array}{c} f_0 \\ g_1 \end{array}\right], \left[\begin{array}{c} f_1 \\ g_1 \end{array}\right], \left[\begin{array}{c} f_1 \\ g_0 \end{array}\right]\colon X \to C_1 \times C_2$$
factor through~$r$, then so does the induced morphism
$$\left[\begin{array}{c} f_0 \\ g_0 \end{array}\right]\colon X \to C_1 \times C_2.$$
If the considered category is an algebraic category, this happens if and only if the relation is difunctional~\cite{Riguet1948}, i.e., it satisfies
$$[x_0Ry_1 \wedge x_1Ry_1 \wedge x_1Ry_0]\Longrightarrow x_0Ry_0$$
for every elements $x_0,x_1\in C_1$ and $y_0,y_1\in C_2$. In comparison, an internal binary relation $r\colon R\rightarrowtail C^2$ in a finitely complete category is $M$-closed with respect to this matrix if and only if, for each object $X$ and each morphisms $f_0,f_1\colon X\to C$, if the three induced morphisms
$$\left[\begin{array}{c} f_0 \\ f_1 \end{array}\right], \left[\begin{array}{c} f_1 \\ f_1 \end{array}\right], \left[\begin{array}{c} f_1 \\ f_0 \end{array}\right]\colon X \to C^2$$
factor through~$r$, then so does the induced morphism
$$\left[\begin{array}{c} f_0 \\ f_0 \end{array}\right]\colon X \to C^2.$$
So the matrix class $\mathsf{mclex}\{M\}$ defined by the matrix $M$ above is the collection of Mal'tsev categories~\cite{CarboniPedicchioPirovano1991,CarboniLambekPedicchio1991}. The concept of a Mal'tsev category has occupied a prominent place in categorical algebra since the 1990's --- see for instance~\cite{BorceuxBourn2004,BournGranJacqmin2020} and the references therein.  Algebraic Mal'tsev categories are nothing other than categories of algebras in a Mal'tsev variety in the sense of~\cite{Smith1976}: a variety containing a ternary term $p$ satisfying
$$\left\{\begin{array}{c}p(x_0,x_1,x_1)=x_0,\\ p(x_1,x_1,x_0)=x_0.\end{array}\right.$$
Such a term can be created in varieties of (not necessarily abelian) group-like structures:
$$p(x,y,z)=x-y+z.$$ 
Mal'tsev varieties are also known in universal algebra as `congruence-permutable' varieties, since they are exactly those varieties in which composition of congruences on an algebra is commutative. This goes back to~\cite{Maltsev1954}, where Mal'tsev varieties as well as Mal'tsev conditions in general were born. 
\end{example}

\begin{remark}\label{RemA}
Expression of the property defining Mal'tsev varieties in terms of congruences recalled in Example~\ref{ExaA} extends to all matrix properties. This is not particularly relevant for the present paper, so we will not elaborate on this in detail. It is worth remarking, however, that such reformulations of matrix properties give rise to geometric interpretations of these properties in the style of~\cite{Smith1976} and~\cite{Gumm1983}. Let us briefly describe this representation, as it opens up potential links with finite geometry. For the following reformulation to hold, we restrict ourselves here to the case where each row of the matrix contains at least one~$0$. Think of each column of the matrix as a point in a discrete plane. Think of each row as an equivalence class of `parallel' (discrete) lines (where lines being parallel is no longer a property, but an imposed structure). If two entries $x_{ij}$ and $x_{ij'}$ in the same row are equal, interpret this as the points $j$ and $j'$ having a line passing through them which belongs to the equivalence class $i$ of parallel lines. In the case of the Mal'tsev matrix (i.e., the matrix from Example~\ref{ExaA}), we get a familiar geometric representation of the property defining a Mal'tsev variety
$$\begin{tikzpicture}
  \filldraw (0,0) circle (2pt) -- (1,0) circle (2pt) node [sloped,midway] {\tiny/};
  \draw (-0.5,1) circle (2pt) -- (0.5,1) circle (2pt) node [sloped,midway] {\tiny/};
  \filldraw (0,0) circle (2pt) -- (-0.5,1) circle (2pt) node [sloped,midway] {\tiny//};
  \draw (1,0) circle (2pt) -- (0.5,1) circle (2pt) node [sloped,midway] {\tiny//};
\end{tikzpicture}$$
where the hollow point represents the extended column of the matrix. The following table links this representation with the matrix:
$$\begin{array}{c|ccc|c} & \bullet & \bullet & \bullet & \circ\\\hline / & 0 & 1 & 1 & 0\\  // & 1 & 1 & 0 & 0 \end{array}$$
Note that for other matrices, such drawings may require curved lines, as the geometries that arise here are in most cases not embeddable in the usual Euclidean geometry. The drawings suggest what the corresponding condition on congruences is: simply view each point as an element of the algebra on which the congruences are defined and each line as the property that the endpoints are in the same class of a congruence, with two lines being `parallel' indicating that the same congruence is considered; the condition then requires that the picture can be completed with an element representing the extended column of the matrix.
\end{remark}

\begin{example}\label{ExaB}
The matrix
$$M=\left[\begin{array}{ccc} 1 & 0 & 0\\ 0 & 1 & 0 \\ 0 & 0 & 1 \end{array}\right]$$
defines majority categories in the sense of~\cite{Hoefnagel2019a, Hoefnagel2020a}, which generalise varieties having a majority term~\cite{Pixley1963}, i.e., a ternary term $p$ satisfying 
$$\left\{\begin{array}{c}p(x_1,x_0,x_0)=x_0,\\ p(x_0,x_1,x_0)=x_0,\\ p(x_0,x_0,x_1)=x_0.\end{array}\right.$$
If Mal'tsev varieties capture varieties of group-like structures, varieties with a majority term capture varieties of lattice-like structures. A majority term can be defined in any lattice by setting:
$$p(x,y,z)=(x\wedge y)\vee (x\wedge z)\vee(y\wedge z).$$
\end{example}

\begin{example}\label{ExaC}
Arithmetical categories in the sense of~\cite{Pedicchio1996} are those Barr-exact categories~\cite{BarrGrilletOsdol1971} with coequalisers in which all ternary relations are $M$-closed, where 
$$M=\left[\begin{array}{ccc} 0 & 1 & 1\\ 1 & 1 & 0 \\ 0 & 1 & 0 \end{array}\right].$$
In the more general context of regular categories~\cite{BarrGrilletOsdol1971} in the place of Barr-exact categories with coequalisers, these become equivalence distributive Mal'tsev categories in the sense of~\cite{GranRodeloNguefeu2020}, where the link with strict $M$-closedness, for the $M$ above, is established. It is easy to prove (and it will be done in Section~\ref{sec:context sensitivity} as an application of our algorithm) that further extension to the finitely complete context gives us the matrix class of Mal'tsev majority categories, already considered in~\cite{Hoefnagel2019a}. In this paper, we thus refer to arithmetical categories as (finitely complete) Mal'tsev majority categories. The universal-algebraic origin of the matrix class of arithmetical categories --- the collection of arithmetical varieties in the sense of~\cite{Pixley1963,Pixley1971,Pixley1972}, can be described in any of the following equivalent ways (among many others):
\begin{itemize}
\item as varieties that contain a \emph{Pixley term}, i.e., a ternary term $p$ satisfying 
$$\left\{\begin{array}{c}p(x_0,x_1,x_1)=x_0,\\ p(x_1,x_1,x_0)=x_0,\\ p(x_0,x_1,x_0)=x_0,\end{array}\right.$$

\item as those Mal'tsev varieties where algebras have distributive congruence lattices,

\item as those Mal'tsev varieties that contain a majority term.
\end{itemize}
Varieties containing Boolean algebra operations are arithmetical (a Boolean algebra is both a group-like and a lattice-like structure). The term $p$ above can be expressed using the Boolean operations as follows:
$$p(x,y,z)=(x\wedge y\wedge z)\vee(x\wedge\neg y)\vee (z\wedge\neg y).$$
\end{example}

\begin{remark}\label{RemB}
In Examples~\ref{ExaA}, \ref{ExaB} and~\ref{ExaC}, the Mal'tsev conditions arise from the matrix properties through a general process described in the Introduction. The Mal'tsev conditions arising in this way are of a very particular type, where each identity in the Mal'tsev condition is of the form $p(x_1,\dots,x_m)=y_1$, where $x_1,\dots,x_m,y_1$ are variables. Many other Mal'tsev conditions studied in universal algebra, which do not have this type, can be strengthened to a Mal'tsev condition of this type. This process can be referred to as \emph{syntactical refinement} (see~\cite{JanelidzePhD}). We illustrate it on the following example. Consider the Mal'tsev condition from~\cite{ChajdaEigenthalerLanger2003} that characterises Mal'tsev varieties with directly decomposable congruences~\cite{FraserHorn1970}. This Mal'tsev condition states that the variety contains binary terms $s_1,\dots,s_m$ and $t_1,\dots,t_m$, and a term $u$ of arity $m+1$, such that the following identities hold in the variety (see Corollary~11.0.6 in~\cite{ChajdaEigenthalerLanger2003}):
$$
\left\{\begin{array}{l} u(x,s_1(x,y),\dots,s_m(x,y))=x,\\
u(y,t_1(x,y),\dots,t_m(x,y))=x,\\
u(y,s_1(x,y),\dots,s_m(x,y))=x,\\
u(x,t_1(x,y),\dots,t_m(x,y))=y.
\end{array}\right. $$
Now we will consider the case when all the binary terms $s_i$ and $t_j$ are actually variables, i.e., they satisfy $v(x,y)=x$ or $v(x,y)=y$. Moreover, we consider the case where $\{(s_i(x,y),t_i(x,y))\,|\,1\leqslant i \leqslant m \}=\{x,y\} \times \{x,y\}$. Writing out the corresponding matrix and deleting the duplicate columns, we get:
$$\left[\begin{array}{ccccc|c} x & x & x & y & y & x\\
y & x & y & x & y & x\\
y & x & x & y & y & x\\
x & x & y & x & y & y
\end{array}\right].$$
To get it in the integer form, we first want to swap $x$ and $y$ in the last row. Then, writing $0$ for $x$ and $1$ for $y$, and deleting the extended column of $0$'s, we get:
$$M=\left[\begin{array}{ccccc} 0 & 0 & 0 & 1 & 1 \\
1 & 0 & 1 & 0 & 1 \\
1 & 0 & 0 & 1 & 1 \\
1 & 1 & 0 & 1 & 0
\end{array}\right].$$
Many other Mal'tsev conditions encountered in universal algebra refine to Mal'tsev conditions given by matrix properties. In some other examples, an undetermined number of `outer terms' $u_1,\dots,u_s$ are considered in the original Mal'tsev condition (instead of just one $u$ as in the example above). In this case, the refinement is usually obtained by considering only the particular case $s=1$ and then by applying to it a similar technique as above. Note that the process of refinement described here is a purely syntactical procedure, and may depend on the choice of presentation of a Mal'tsev condition.
\end{remark}

It is clear that the matrix property arising from $M\in\M(n,m,k)$ is the same as the one arising from $M$ viewed as a member of $\M(n,m,k')$, for any $k'\geqslant k$. Moreover, according to Proposition~1.7 in~\cite{ZJanelidze2006b}, we have:
\begin{itemize}
\item Given two matrices $M\in\M(n,m,k)$ and $N\in\M(n,m',k')$ such that every column of $M$ is a column of~$N$, then any finitely complete category with $M$-closed relations also has $N$-closed relations, i.e., $\mathsf{mclex}\{M\} \subseteq \mathsf{mclex}\{N\}$.
\item Given two matrices $M\in\M(n,m,k)$ and $N\in\M(n',m,k')$ such that every row of $N$ is a row of~$M$, then any finitely complete category with $M$-closed relations also has $N$-closed relations, i.e., $\mathsf{mclex}\{M\} \subseteq \mathsf{mclex}\{N\}$.
\end{itemize}
The first of these results can be proved by showing that an internal $n$-ary strictly $M$-closed relation is strictly $N$-closed (for $M$ and $N$ as in the statement). The other result can be proved from the following fact. Using the notation in the statement, there exists a function $g\colon\{1,\dots,n'\}\to\{1,\dots,n\}$ such that, for each $i\in\{1,\dots,n'\}$, the $i$-th row of $N$ is the same as the $g(i)$-th row of~$M$. Given an object $C$ in a finitely complete category, we can consider the $n$-th and $n'$-th powers of $C$ with respective projections denoted by $\pi_1,\dots,\pi_n\colon C^n\to C$ and $\pi'_1,\dots,\pi'_{n'}\colon C^{n'}\to C$. The function $g$ induces a unique morphism $\overline{g}\colon C^n\to C^{n'}$ such that $\pi'_i\overline{g}=\pi_{g(i)}$ for each $i\in\{1,\dots,n'\}$. The second statement follows from the fact that, given an internal $n'$-ary relation $r\colon R\rightarrowtail C^{n'}$ on $C$ and considering the pullback
$$\vcenter{\xymatrix{S \pb \ar[r] \ar@{ >->}[d]_-{s} & R \ar@{ >->}[d]^-{r} \\ C^n \ar[r]_-{\overline{g}} & C^{n'}}}\quad ,$$
the internal $n$-ary relation $s$ is $M$-closed if and only if $r$ is $N$-closed. It follows from these results that matrix properties of finitely complete categories are invariant under duplication and permutation of the columns and of the rows of the matrices.

The notions of $M$-closedness and strict $M$-closedness have been studied in~\cite{ZJanelidze2006a,ZJanelidze2006b}. In this paper, we add a third type of closedness property of a relation under a matrix $M\in\M(n,m,k)$. Call an internal $n'$-ary relation in a category \emph{$M$-sharp} if it is strictly $M'$-closed under any matrix $M'\in\M(n',m,k)$, obtained from $M$ by permutation, duplication and deletion of rows --- in other words, every row of $M'$ is a row of~$M$. Note that in this notion, the matrix $M$ itself does not need to have exactly $n'$ rows --- one can have $n>n'$, $n=n'$ or $n<n'$. By the above remark, a finitely complete category has $M$-closed relations if and only if every $n'$-ary internal relation in it is $M$-sharp, where $n'$ can be arbitrary or alternatively, any fixed value that is greater or equal to~$n$. As an illustration of this new closedness property of a relation, consider an affine subspace $R$ of a three dimensional vector space $F^3$ over a field~$F$, seen as a ternary relation on~$F$. Such an $R$ is $M$-sharp for the Mal'tsev matrix
\begin{equation}\label{equ Maltsev matrix}
M=\left[\begin{array}{ccc} 0 & 1 & 1\\ 1 & 1 & 0 \end{array}\right]
\end{equation}
for the following reason. There are $8$ possible $M'$-s for this $M$: 
$$
\begin{array}{cccc}
M_{111}=\left[\begin{array}{ccc} 0 & 1 & 1\\ 0 & 1 & 1\\ 0 & 1 & 1\end{array}\right] &  
M_{112}=\left[\begin{array}{ccc} 0 & 1 & 1\\ 0 & 1 & 1\\ 1 & 1 & 0\end{array}\right] &
M_{121}=\left[\begin{array}{ccc} 0 & 1 & 1\\ 1 & 1 & 0\\ 0 & 1 & 1\end{array}\right] &
M_{122}=\left[\begin{array}{ccc} 0 & 1 & 1\\ 1 & 1 & 0\\ 1 & 1 & 0\end{array}\right]\\ \\
M_{211}=\left[\begin{array}{ccc} 1 & 1 & 0\\ 0 & 1 & 1\\ 0 & 1 & 1\end{array}\right] &  
M_{212}=\left[\begin{array}{ccc} 1 & 1 & 0\\ 0 & 1 & 1\\ 1 & 1 & 0\end{array}\right] &
M_{221}=\left[\begin{array}{ccc} 1 & 1 & 0\\ 1 & 1 & 0\\ 0 & 1 & 1\end{array}\right] &
M_{222}=\left[\begin{array}{ccc} 1 & 1 & 0\\ 1 & 1 & 0\\ 1 & 1 & 0\end{array}\right]
\end{array}$$
Row-wise interpretations of, say, $M_{212}$ of type $(F,F,F)$ are given by matrices of elements of $F$ of the form
$$\left[\begin{array}{ccc|c} a & a & b & b\\ c & d & d & c \\ e & e & f & f\end{array}\right].$$
Whenever the `left columns' of this matrix (i.e., columns left to the vertical line) belong to the affine subspace, so does the `right column' (the column right to the vertical line) thanks to the fact that 
$$\left[\begin{array}{c} a \\ c \\ e \end{array}\right]-\left[\begin{array}{c} a \\ d \\ e \end{array}\right]+\left[\begin{array}{c} b \\ d \\ f \end{array}\right]=\left[\begin{array}{c} b \\ c \\ f \end{array}\right].$$
For the same reason, $R$ will be strictly closed with respect to the remaining seven matrices. As a counter-example, let us remark that the ternary relation
$$R'=\left\{\left[\begin{array}{c} x \\ y \\ y \end{array}\right], \left[\begin{array}{c} y \\ y \\ y \end{array}\right], \left[\begin{array}{c} y \\ x \\ x \end{array}\right]\right\} \subset \{x,y\}^3$$
on the two element set $\{x,y\}$ is strictly $M_{221}$-closed but not (strictly) $M_{122}$-closed. It is thus neither $M$-sharp nor $M_{221}$-sharp as every row of $M_{122}$ is a row of $M_{221}$. However, one trivially has that $R'$ is $M_{111}$-sharp and $M_{222}$-sharp since both these matrices have a column of zeros.

It will be useful to have a designated concept of an interpretation for the closedness property of sharpness. Given a sequence of sets $(S_1,\dots,S_{n'})$, we define a \emph{reduction} of type $(S_1,\dots,S_{n'})$ of a matrix $M\in\M(n,m,k)$ to be a matrix $M'''$ which can be obtained as follows:
\begin{itemize}
\item choose some matrix $M'\in\M(n',m,k)$ whose every row is a row in~$M$;
\item choose some row-wise interpretation $M''$ of type $(S_1,\dots,S_{n'})$ of~$M'$;
\item possibly duplicate and permute some (left) columns of $M''$ and possibly delete some duplicate (left) columns of $M''$ to obtain the reduction~$M'''$.
\end{itemize}
Then an internal $n'$-ary relation between the objects $C_1,\dots,C_{n'}$ in a category $\mathbb{C}$ is $M$-sharp if and only if it is compatible with every reduction of $M$ of type $(\mathbb{C}(X,C_1),\dots,$ $\mathbb{C}(X,C_{n'}))$ (for an arbitrary object~$X$). We say that a matrix $M'''$ is a \emph{reduction} of $M$ if it is a reduction of type $(S_1,\dots,S_{n'})$ of $M$ for some sets $S_1,\dots,S_{n'}$. Let us give a concrete example of a reduction of type $(\{0,1\},\{2,3\},\{4,5,6\})$ of the Mal'tsev matrix $M$ as in~(\ref{equ Maltsev matrix}). As a matrix $M'$ with three rows each of which is a row of~$M$, we can choose, e.g., the matrix $M_{121}$ as above. As a row-wise interpretation $M''$ of type $(\{0,1\},\{2,3\},\{4,5,6\})$ of $M_{121}$, we can choose, e.g., the matrix
$$M''=\left[\begin{array}{ccc|c} 1 & 0 & 0 & 1\\ 3 & 3 & 3 & 3 \\ 4 & 6 & 6 & 4\end{array}\right].$$
Finally, the matrix
$$M'''=\left[\begin{array}{ccc|c} 0 & 1 & 1 & 1\\ 3 & 3 & 3 & 3 \\ 6 & 4 & 4 & 4\end{array}\right]$$
which is obtained from $M''$ by deleting the second column, duplicating the first column, and then reversing the left columns, is an example of a reduction of type $(\{0,1\},\{2,3\},$ $\{4,5,6\})$ of~$M$.

Let us conclude this section with a useful lemma on reductions.

\begin{lemma}\label{Lemma reductions}
Given integers $n,n',k,k'>0$ and $m\geqslant 0$ and a matrix $M\in\M(n,m,k)$, if $M'''$ is a reduction of $M$ of type $(\{0,\dots,k'-1\},\dots,\{0,\dots,k'-1\})$ whose right column is $[0]_{n'}$, then every finitely complete category with $M$-closed relations also has $M'''$-closed relations, i.e., $\mathsf{mclex}\{M\} \subseteq \mathsf{mclex}\{M'''\}$ (considering only the left part of~$M'''$).
\end{lemma}

\begin{proof}
Let $M'$ and $M''$ be matrices as in the above definition of a reduction. The inclusion $\mathsf{mclex}\{M\} \subseteq \mathsf{mclex}\{M'\}$ and the equality  $\mathsf{mclex}\{M''\} = \mathsf{mclex}\{M'''\}$ have been discussed above. It thus remains to show that the inclusion $\mathsf{mclex}\{M'\} \subseteq \mathsf{mclex}\{M''\}$ holds. But this follows from the fact that, by definition of $M''$ and since its right column is~$[0]_{n'}$, each row-wise interpretation of $M''$ is also a row-wise interpretation of $M'$ of the same type with the same right column.
\end{proof}

\section{Triviality and functionality}\label{sec2}

A matrix $M\in\M(n,m,k)$ is said to be \emph{trivial} if every finitely complete category with $M$-closed relations is a preorder, i.e., any two parallel morphisms in the category coincide. It is not difficult to see that any preorder has $M$-closed relations for any non-empty matrix~$M$, so:
\begin{enumerate}[label=(\roman*)]
    \item\label{triviality first observation} Non-empty trivial matrices are precisely those matrices $M$ for which the matrix class $\mathsf{mclex}\{M\}$ of finitely complete categories with $M$-closed relations matches with the collection of finitely complete preorders.
\end{enumerate}
To say something about the empty matrices, first let us set what we mean by an `empty matrix'. Formally, elements of each $\M(n,m,k)$ are functions $\{1,\dots,n\}\times\{1,\dots,m\}\to\{0,\dots,k-1\}$. Since $n$ and $k$ are considered to be positive, empty matrices arise when $m=0$. Applying the definition of $M$-closedness to this case, we get: 
\begin{enumerate}[resume, label=(\roman*)]
    \item\label{triviality m=0} When $m=0$, since matrix properties on finitely complete categories are invariant under duplication of rows, a finitely complete category $\mathbb{C}$ has $M$-closed relations if and only if for any monomorphism $r\colon R\rightarrowtail C$ and for any morphism $x\colon X\to C$, we have $ru=x$ for some morphism~$u$. This is equivalent to saying that any monomorphism in $\mathbb{C}$ is an isomorphism (just consider the case when $x=1_C$). Considering the equaliser of two parallel morphisms, this implies that such a category is a preorder. Therefore, any morphism is a monomorphism, and thus an isomorphism. Considering the product of two objects $X$ and~$Y$, we know that $\mathbb{C}(X,Y)$ is not empty and is thus a singleton set. Therefore, if $m=0$, $\mathsf{mclex}\{M\}$ is the collection of categories equivalent to the single morphism category, i.e, the terminal category.
\end{enumerate}
With the observation~\ref{triviality m=0}, we have completely described the matrix classes of finitely complete categories determined by empty matrices. Moreover, observations~\ref{triviality first observation} and~\ref{triviality m=0} combine to the following:

\begin{lemma}\label{LemA}
A matrix $M$ is trivial if and only if $\mathsf{mclex}\{M\}$ is the collection of all finitely complete preorders or the collection of preorders with a single isomorphism class of objects, the second case occurring if and only if $M$ is empty.
\end{lemma}

Let us recall that a (non-row-wise) interpretation of type $S$ of a matrix $A=[a_{ij}]_{i,j}\in \M(n,m,k)$ is given by the following data:
\begin{enumerate}[label=(\roman*)]
    \item\label{interpretation function f} a function $f\colon \{0,\dots,k-1\}\to S$,
    
    \item an $n\times m$ matrix $B=[b_{ij}]_{i,j}$ with entries from~$S$,
    
    \item these two ingredients being related as follows: $f(a_{ij})=b_{ij}$ for each $i\in\{1,\dots,n\}$ and each $j\in\{1,\dots,m\}$.
\end{enumerate}
This data can be represented as a commutative triangle of maps in~$\mathbf{Set}$
$$\xymatrix{
 & n S^k\ar[rd]^-{\pi^S_A} & \\
n\mathbf{1}\ar[rr]_-{B}\ar[ur]^-{nf} & & S^m
}$$
where:
\begin{itemize}
\item $nX$ stands for the $n$-fold sum (i.e., the coproduct or the disjoint union) of a set $X$ with itself and $X^p$ for the dual (i.e., the $p$-fold product),

\item $\mathbf{1}$ is the terminal object (i.e., the singleton set),

\item the morphism $B$ is then the obvious representation of the matrix $B$ as a map (the values of this map are the $m$-tuples formed by the rows of~$B$),

\item $f$ is a map $\mathbf{1}\to S^k$ that has its unique value the function $f$ from~\ref{interpretation function f} seen as a $k$-tuple of elements of~$S$, while $nf$ is the canonically induced map between $n$-fold sums of sets,

\item the map $\pi_A^S$ is the canonical map from the sum to a product given by the matrix
$$\pi_A^S=\left[\begin{array}{ccc} \pi_{a_{11}+1} & \dots & \pi_{a_{1m}+1} \\ \vdots & & \vdots \\ \pi_{a_{n1}+1} & \dots & \pi_{a_{nm}+1} \end{array}\right]$$ of product projections $S^k\to S$ ($\pi_1$ is the first product projection, $\pi_2$ the second, and so forth). 
\end{itemize}
Notice that $\pi_A^S$ itself represents an interpretation of $A$ of type $\mathbf{Set}(S^k,S)$. Thus, we have a specific interpretation of~$A$, given by $\pi_A^S$, which `generates' all interpretations via the triangle above. The image of the map $\pi_A^S$ is nothing other than the set of all possible rows of all possible row-wise interpretations of $A$ of type $(S,\dots,S)$.

The morphism $\pi_A^S$ can obviously be defined in any category $\mathbb{C}$ having the required sums and products --- we call $\pi_A^S$ a \emph{canonical interpretation} of $A$ in $\mathbb{C}$ of type~$S$. Consider a triangle as above in a category~$\mathbb{C}$, with $n\mathbf{1}$ replaced by $nC$, the object $S$ renamed to $X$ and the matrix $A$ renamed to~$M$:
$$\xymatrix{ & n X^k\ar[rd]^-{\pi^X_M} & \\ nC\ar[rr]_-{B}\ar[ur]^-{nf} & & X^m}$$ This triangle now describes interpretations of $M$ of type $\mathbb{C}(C,X)$. Therefore, an internal $n$-ary relation on an object $C$ in the dual category $\mathbb{C}^\mathsf{op}$, described in $\mathbb{C}$ as an epimorphism
$$r=\left[\begin{array}{c} r_1 \\ \vdots \\ r_n \end{array}\right]\colon nC\twoheadrightarrow R,$$  is $M$-closed over $X$ if and only if  
every commutative diagram of solid arrows in the following display can always be filled with the dashed morphism retaining the commutativity:
\begin{equation}\label{EquA}
\vcenter{\xymatrix{X & &\\ & & n X^k\ar[rd]^-{\pi^X_M}\ar[llu]_-{\pi_{[0]_n}^X} & \\ & nC\ar@{->>}[dl]_{r}\ar[ur]_-{nf} & & X^m \\ R\ar[urrr]\ar@{-->}[uuu] & }}
\end{equation}
If we let $C=X^k$ and $f$ be the identity morphism of~$X^k$, then $nf$ is the identity morphism of $nX^k$ and so the diagrammatic condition above becomes:
\begin{equation}\label{EquB}
\vcenter{\xymatrix{X & &\\ & & n X^k\ar@{->>}[ld]_-{r'}\ar[rd]^-{\pi^X_M}\ar[llu]_-{\pi_{[0]_n}^X} & \\ & R'\ar@{-->}[luu]\ar[rr] & & X^m }}
\end{equation}
So, fixing a matrix $M\in\M(n,m,k)$, a category $\mathbb{C}$ with finite products and finite colimits and an object $X$ in~$\mathbb{C}$, if every internal $n$-ary relation on any object in $\mathbb{C}^\mathsf{op}$ is $M$-closed over~$X$, then any commutative diagram~(\ref{EquB}) in $\mathbb{C}$ (for an arbitrary epimorphism $r'\colon nX^k\twoheadrightarrow R'$) can be filled in with a dashed morphism. Moreover, by pushing out $r$ along $nf$ in~(\ref{EquA}) we can get the converse implication (note that a pushout of an epimorphism is an epimorphism). Suppose furthermore that any morphism $h$ in $\mathbb{C}$ has a \emph{universal epi-factorisation}, i.e., a factorisation $h=ge$ via an epimorphism $e$ such that any similar factorisation $h=g'e'$ with $e'$ an epimorphism yields $e=ue'$ for a (necessarily unique) morphism~$u$. Then we can assume the bottom triangle in the diagram~(\ref{EquB}) to be such a universal factorisation. Existence of the dashed morphism for it alone will imply the existence of the dashed morphism when $r'$ is an arbitrary epimorphism. If furthermore kernel pairs (i.e., pullbacks of a morphism along itself) exist, then the morphism $R'\to X^m$ in the universal epi-factorisation of the morphism $\pi^X_M$ is necessarily a monomorphism (consider the epimorphism obtained by composing $r'$ with the coequaliser of the kernel pair of the morphism $R'\to X^m$). In the case when $\mathbb{C}=\mathbf{Set}$, this monomorphism is given by the image of the map~$\pi^X_M$. In this case, what the dashed map $R'\to X$ does is to assign to each row $(x_{i1},\dots,x_{im})$ from some interpretation of $M$ of type~$X$, the corresponding value $y_i$ in the extended display of the interpretation of~$M$:
$$\left[\begin{array}{ccc|c} x_{11} & \dots & x_{1m} & y_1\\ \vdots & & \vdots & \vdots\\ x_{n1} & \dots & x_{nm} & y_n \end{array}\right]$$
Of course, such a map exists if and only if $y_i$ never depends on the choice of the interpretation. Let us say that such a matrix $M$ is \emph{functional} in~$X$. Keeping the same terminology in the general case, we obtain: 

\begin{theorem}\label{ThmA}
Given integers $n,k>0$ and $m\geqslant 0$, a matrix $M\in\M(n,m,k)$ and an object $X$ in a category $\mathbb{C}$ having finite products and finite colimits, as well as universal epi-factorisations, the following conditions are equivalent:
\begin{enumerate}[label=(\arabic*)]
\item Every internal $n$-ary relation on any object in $\mathbb{C}^\mathsf{op}$ is $M$-closed over~$X$.

\item Every internal $n$-ary relation in $\mathbb{C}^\mathsf{op}$ is strictly $M$-closed over~$X$.

\item $M$ is functional in~$X$, i.e., for the universal epi-factorisation $\pi^X_M=i^X_Mr^X_M$ of the canonical interpretation $\pi_M^X$ of $M$ in $\mathbb{C}$ of type~$X$, there is a morphism $p^X_M\colon R_M^X\to X$ making the diagram $$\xymatrix{
X & &\\ & & n X^k\ar@{->>}[ld]_-{r^X_M}\ar[rd]^-{\pi^X_M}\ar[llu]_-{\pi_{[0]_n}^X} & \\
& R_M^X\ar@{-->}[luu]^-{p^X_M}\ar[rr]_-{i^X_M} & & X^m
}$$
commute.
\end{enumerate}
\end{theorem}

Note that $M$ is \emph{dysfunctional} in a set~$X$ (i.e., it is not functional in $X$ as an object in $\mathbf{Set}$) if and only if there are two interpretations of $M$ of type $X$ having a common row (not necessarily in the same position), with the value in the corresponding extended column being different. This is equivalent to existence of a reduction of $M$ of the form
$$T=\left[\begin{array}
{ccc|c} a_1 & \dots & a_m & b\\ a_1 & \dots & a_m & c \end{array}\right]$$
where $b\neq c$ and all entries are elements of~$X$. Thus, when $X$ is either empty or singleton, any matrix $M$ is automatically functional in $X$. As we will shortly see, there is a simple way of verifying that a matrix is dysfunctional in a set having at least two elements that does not require modifying the matrix. It amounts to finding a row that does not contain $0$ as an entry, or if there is no such row, then finding two distinct rows and $0$ entries in each of the rows which cannot be connected to each other by a path along the matrix consisting of steps that alternate between vertical steps of switching between the two rows in the same column and horizontal steps that move within the selected row from one position to another position having the same entry as at the original position. For instance, such a path is possible between the $0$ entries of the two rows of the Mal'tsev matrix:
$$\xymatrix{ 0\ar@{-}[d] & 1\ar@{-}[r] & 1\ar@{-}[d]\\ 1\ar@{-}[r] & 1\ar@{-}[u] & 0 }$$
This implies that a reduction as described above does not exist. Indeed, since each row of the Mal'tsev matrix contains a $0$ entry, we cannot get $T$ by using just one row of the Mal'tsev matrix. An attempt to get $T$ by using both rows of the Mal'tsev matrix would also fail, since equalising the left columns of the Mal'tsev matrix forces the right extended column to have equal entries too:
$$\left[\begin{array}{ccc|c} 0 & 1 & 1 & 0\\ 1 & 1 & 0 & 0 \end{array}\right]\to \left[\begin{array}{ccc|c} a & 1 & 1 & a\\ a & a & 0 & 0 \end{array}\right]\to \left[\begin{array}{ccc|c} a & a & a & a\\ a & a & 0 & 0 \end{array}\right]\to \left[\begin{array}{ccc|c} a & a & a & a\\ a & a & a & a \end{array}\right] .$$ 
We will call such a path in a matrix a linkage. More precisely, for an integer $l\geqslant 0$, a \emph{linkage} of length~$l$ connecting the positions $(i_0,j_0)$ and $(i_l,j_l)$ in a matrix $M$ is a sequence
$$(i_0,j_0),\dots,(i_l,j_l)$$
of $l+1$ positions in the matrix $M$ such that
\begin{itemize}
\item for each $t\in\{0,\dots,l-1\}$ with $t\equiv 0 \,\,(\mod 4)$, one has $i_t=i_0$, $i_{t+1}=i_l$ and $j_t=j_{t+1}$ (i.e., $(i_t,j_t)\to (i_{t+1},j_{t+1})$ is a vertical step from the $i_0$-th row to the $i_l$-th row),
\item for each $t\in\{0,\dots,l-1\}$ with $t\equiv 1 \,\,(\mod 4)$, one has $i_t=i_{t+1}=i_l$ and the entries of $M$ in the positions $(i_t,j_t)$ and $(i_{t+1},j_{t+1})$ are equal (i.e., $(i_t,j_t)\to (i_{t+1},j_{t+1})$ is a horizontal step in the $i_l$-th row),
\item for each $t\in\{0,\dots,l-1\}$ with $t\equiv 2 \,\,(\mod 4)$, one has $i_t=i_l$, $i_{t+1}=i_0$ and $j_t=j_{t+1}$  (i.e., $(i_t,j_t)\to (i_{t+1},j_{t+1})$ is a vertical step from the $i_l$-th row to the $i_0$-th row),
\item for each $t\in\{0,\dots,l-1\}$ with $t\equiv 3 \,\,(\mod 4)$, one has $i_t=i_{t+1}=i_0$ and the entries of $M$ in the positions $(i_t,j_t)$ and $(i_{t+1},j_{t+1})$ are equal (i.e., $(i_t,j_t)\to (i_{t+1},j_{t+1})$ is a horizontal step in the $i_0$-th row).
\end{itemize}
A linkage of length $0$ is thus just a position in the matrix. A linkage of length $1$ connects two positions in the same column of two rows, and so on. So linkages can be represented by the shapes displayed below (up to permutation and duplication of columns) and ordered by increasing length.
$$\xymatrix{ a_0 \\ },\quad \xymatrix{ a_0\ar@{-}[d] \\ a_1},\quad \xymatrix{ a_0\ar@{-}[d] & \\ a_1\ar@{-}[r] & a_1},\quad \xymatrix{ a_0\ar@{-}[d] & a_2\ar@{-}[d] \\ a_1\ar@{-}[r] & a_1},\quad \xymatrix{ a_0\ar@{-}[d] & a_2\ar@{-}[d]\ar@{-}[r] & a_2\\ a_1\ar@{-}[r] & a_1 &},\quad \xymatrix{ a_0\ar@{-}[d] & a_2\ar@{-}[d]\ar@{-}[r] & a_2\ar@{-}[d]\\ a_1\ar@{-}[r] & a_1 & a_3},\cdots$$
According to our definition, a linkage of positive length must start with a vertical step; but up to starting with the path
$$(i_0,j_0),(i_l,j_0),(i_l,j_0),(i_0,j_0),(i_0,j_1)$$
one may consider horizontal steps as first `effective' steps. The reader may have noticed that the essence of what is going on here is the following. For a matrix $M\in\M(n,m,k)$, each row gives rise to an equivalence relation on the set $\{1,\dots,m\}$ --- the kernel relation of the row seen as a map $\{1,\dots,m\}\to\{0,\dots,k-1\}$ (this relates to the geometric interpretation of matrix properties from Remark~\ref{RemA}). The positions $(i,j)$ and $(i',j')$ are connected by a linkage if and only if $j$ and $j'$ fall in the same equivalence class of the join of the two equivalence relations given by the two rows $i$ and~$i'$. Thus, linkage partitions the columns of $M$ into equivalence classes --- we call these the \emph{linkage classes} for the pair of rows. We can see that two positions in two rows are connected by a linkage if and only if, in every two interpretations of the two rows matching in all entries except possibly the extended entry, the positions have the same entries. Indeed, while the `only if part' follows immediately from the definition of a linkage, the `if part' can be seen by interpreting each entry of each of the two rows by the corresponding linkage class in the quotient of $\{1,\dots,m\}$ by the above equivalence relation. Moreover, let us remark that, considering two $0$ entries in positions $(i,j)$ and $(i',j')$ of~$M$, if those two entries are connected by a linkage, then any $0$ entry in the $i$-th row is connected by a linkage to any $0$ entry in the $i'$-th row of~$M$.

\begin{theorem}\label{ThmC}
Given integers $n,k>0$ and $m\geqslant 0$ and a matrix $M\in\M(n,m,k)$, the following conditions are equivalent:
\begin{enumerate}[label=(\arabic*)]
\item\label{ThmC 1} $M$ is functional in every set~$X$.

\item\label{ThmC 2} $M$ is functional in a set having at least two elements.

\item\label{ThmC 3} $M$ is functional in the set $\{0,1\}$.

\item\label{ThmC 4} $M$ is non-empty and does not have a reduction given by any of the following two matrices:
$$\left[\begin{array}{c|c} 1 & 0\end{array}\right],\quad \left[\begin{array}{cc|c} 1 & 0 & 0\\ 0 & 1 & 0\end{array}\right].$$

\item\label{ThmC 5} Every row of $M$ has $0$ as one of its entries and any two $0$ entries in any two distinct rows of $M$ can be connected by a linkage.

\item\label{ThmC 6} Given any two rows of $M$, each has $0$ as an entry such that the two positions can be connected by a linkage.

\item\label{ThmC 7} $\mathbf{Set}^\mathsf{op}$ has $M$-closed relations.

\item\label{ThmC 8} $M$ is not a trivial matrix.
\end{enumerate} 
\end{theorem}

\begin{proof}
Let us first notice that if $M$ is empty (i.e., if $m=0$), all statements are false and thus are equivalent. Let us now suppose that $M$ is non-empty. The implications \ref{ThmC 1}$\Rightarrow$\ref{ThmC 3}$\Rightarrow$\ref{ThmC 2} are obvious. The following argument proves simultaneously that if \ref{ThmC 5} does not hold then \ref{ThmC 4} does not hold and the latter implies that \ref{ThmC 2} does not hold, giving us the implications \ref{ThmC 2}$\Rightarrow$\ref{ThmC 4}$\Rightarrow$\ref{ThmC 5}. Suppose \ref{ThmC 5} does not hold. If $M$ has a (non-empty) row where $0$ is not an entry, then $$\left[\begin{array}{c|c} 1 & 0\end{array}\right]$$ is a reduction of $M$ (select that row of~$M$, interpret $0$ by $0$ and each element of $\{1,\dots,k-1\}$ by $1$ and delete duplicate columns). If $M$ has such a reduction, it also admits
$$\left[\begin{array}{c|c} 1 & 0\\ 1 & 1 \end{array}\right]$$
as a reduction, which breaks functionality of $M$ in any set with at least two elements. If $M$ has two $0$ entries in two distinct rows that are not connected by a linkage, then the columns of $M$ can be divided into two disjoint sets: 
\begin{itemize}
\item The first set is the linkage class of the given position of $0$ in the first row.

\item The second set is the union of all the remaining linkage classes.
\end{itemize}
This guarantees that $M$ has the reduction
$$\left[\begin{array}{cc|c} 1 & 0 & 0\\ 0 & 1 & 0\end{array}\right]$$
(select those two rows, interpret in the first row the entries from the first set by $0$ and the entries from the second set by $1$ and vice-versa for the second row and permute and delete the columns as necessary). From this reduction, we can then get the undesired reduction
$$\left[\begin{array}{cc|c} 1 & 0 & 0\\ 1 & 0 & 1\end{array}\right]$$
violating functionality of $M$ in any set with at least two elements. \ref{ThmC 5}$\Rightarrow$\ref{ThmC 6} is trivial. Next, we prove \ref{ThmC 6}$\Rightarrow$\ref{ThmC 1}. Suppose \ref{ThmC 1} does not hold. Then $M$ has a reduction of the form
$$\left[\begin{array}
{ccc|c} a_1 & \dots & a_m & b\\ a_1 & \dots & a_m & c \end{array}\right]$$
with $b\neq c$. If either $b$ does not occur in the first row or $c$ in the second, then $0$ must not occur in one of the rows of $M$ and so \ref{ThmC 6} gets violated. If this is not the case and $b$ does occur in the first row of this reduction and $c$ in the second row, since the rows are identical and $b$ and $c$ are distinct, all of these occurrences must come from separate linkage classes. Then \ref{ThmC 6} gets violated again and so \ref{ThmC 6}$\Rightarrow$\ref{ThmC 1}. The equivalence \ref{ThmC 1}$\Leftrightarrow$\ref{ThmC 7} follows immediately from Theorem~\ref{ThmA}. Since $\mathbf{Set}^\mathsf{op}$ is not a preorder, the implication \ref{ThmC 7}$\Rightarrow$\ref{ThmC 8} is straightforward from the definition of a trivial matrix. To complete the proof of the theorem, it remains to prove \ref{ThmC 8}$\Rightarrow$\ref{ThmC 4}. If \ref{ThmC 4} does not hold, since we have supposed that $M$ is non-empty, $M$ admits one of the two matrices mentioned in \ref{ThmC 4} as a reduction. Therefore, in view of Lemma~\ref{Lemma reductions}, every finitely complete category with $M$-closed relations has either $M_1$-closed relations or $M_2$-closed relations with
$$M_1=\left[\begin{array}{c} 1 \end{array}\right] \quad \text{and} \quad M_2=\left[\begin{array}{cc} 1 & 0 \\ 0 & 1 \end{array}\right].$$
If $f$ and $g$ are two parallel morphisms
$$\xymatrix{X \ar@<3.5pt>[r]^-{f} \ar@<-3.5pt>[r]_-{g} & Y}$$
in a finitely complete category, the induced morphisms
$$\xymatrix@C=3pc{X \ar@<3.5pt>[r]^-{\left[\begin{array}{c} f \\ f \end{array}\right]} \ar@<-3.5pt>[r]_-{\left[\begin{array}{c} g \\ g \end{array}\right]} & Y^2}$$
to the power $Y^2$ factor through the diagonal
$$\xymatrix@C=5pc{Y \ar@{ >->}[r]^-{\Delta_Y=\left[\begin{array}{c} 1_Y \\ 1_Y \end{array}\right]} & Y^2.}$$
Since $\Delta_Y$ is a monomorphism, if the category has $M_1$-closed relations (and seeing $\Delta_Y$ as a unary relation on~$Y^2$) or if the category has $M_2$-closed relations (and seeing $\Delta_Y$ as a binary relation between $Y$ and~$Y$), the morphism
$$\xymatrix@C=3pc{X \ar[r]^-{\left[\begin{array}{c} f \\ g \end{array}\right]} & Y^2}$$
also factors through $\Delta_Y$ proving that $f=g$ and thus the category is a preorder. This shows that if \ref{ThmC 4} does not hold, $M$ is a trivial matrix, proving the implication \ref{ThmC 8}$\Rightarrow$\ref{ThmC 4}.
\end{proof}

This theorem already gives a minor classification result: 

\begin{corollary}\label{CorB}
There are only two matrix classes given by a non-empty matrix having one row: the matrix class of finitely complete preorders and the matrix class of all finitely complete categories. There is only one other matrix class given by a non-empty matrix with $2$ rows --- the matrix class of Mal'tsev categories. 
\end{corollary}

\begin{proof}
If the unique row of a non-empty matrix $M$ contains~$0$, then every category has $M$-closed relations. If $0$ does not appear in the unique row, then $M$ is trivial by Theorem~\ref{ThmC}. By Lemma~\ref{LemA}, this then gives the matrix class of finitely complete preorders. This proves the first part of the corollary. If a non-empty matrix $M\in\M(2,m,k)$ is non-trivial, then by Theorem~\ref{ThmC}, each of the two rows has an entry $0$ connected to each other by a linkage. The length $l$ of the shortest such linkage cannot be of the form $4l'$ or $4l'+3$, where $l'$ is a non-negative integer, since the two considered positions are not in the same row. It is also easy to see that, by minimality of~$l$, it cannot be of the form $4l'+2$. We thus have $l=4l'+1$ for some integer $l'\geqslant 0$. If $l=1$, then the matrix class is the collection of all finitely complete categories since $M$ admits
$$\left[\begin{array}{c} 0 \\ 0\end{array}\right]$$
as one of its columns. If $l \geqslant 5$, let us consider the reduction of $M$ obtained by substituting $1$ in the place of each non-zero entry and deleting duplicate (left) columns. If the reduced matrix has more columns than the Mal'tsev matrix (Example~\ref{ExaA})
$$\left[\begin{array}{ccc} 0 & 1 & 1 \\ 1 & 1 & 0\end{array}\right]$$
then the only possibility is that it has a column of~$0$'s, which is impossible since by reduction no new $0$ entries were created and this would contradict the minimality of~$l$. So the above matrix is a reduction of $M$ and hence, every category in the matrix class is a Mal'tsev category by Lemma~\ref{Lemma reductions}. Conversely, let us prove that the matrix class contains all Mal'tsev categories. Let us denote by $M'$ the matrix formed by the columns through which the considered shortest linkage
$$(i_0,j_0),(i_1,j_1),\dots,(i_{4l'+1},j_{4l'+1})$$
passes. We thus have $\mathsf{mclex}\{M'\}\subseteq\mathsf{mclex}\{M\}$. This matrix $M'$ is a reduction of the matrix
$$M_{l'}=\left[\begin{array}{ccccccccc} 0 & 1 & 1 & 2 & 2 & \cdots & \cdots & l' & l' \\ 1 & 1 & 2 & 2 & \cdots & \cdots & l' & l' & 0\end{array}\right];$$
this can be seen using the row-wise interpretation given by the functions $$f_1,f_2\colon\{0,\dots,l'\}\to\{0,\dots,k-1\}$$ defined for each $t\in\{0,\dots,l'\}$ by
$$f_1(t)=a_{(i_{4t},j_{4t})}$$
and
$$f_2(t)=\begin{cases} 0 & \text{if }t=0\\ a_{(i_{4t-2},j_{4t-2})} & \text{if }t>0\end{cases}$$
where the $a_{ij}$'s denote the entries of~$M$. Since $f_1(0)=f_2(0)=0$, by Lemma~\ref{Lemma reductions}, we know that $\mathsf{mclex}\{M_{l'}\}\subseteq\mathsf{mclex}\{M'\}$. Finally, using induction on~$l'$, it is easy to prove that each Mal'tsev category has $M_{l'}$-closed relations for each $l'\geqslant 1$. This proves that the matrix class $\mathsf{mclex}\{M\}$ contains all Mal'tsev categories.
\end{proof}

A matrix $M\in\M(n,m,k)$ is said to be \emph{anti-trivial} if any finitely complete category has $M$-closed relations. We invite the reader to compare the equivalence \ref{ThmC 7}$\Leftrightarrow$\ref{ThmC 8} of Theorem~\ref{ThmC} with the equivalence \ref{anti-trivial 2}$\Leftrightarrow$\ref{anti-trivial 3} of the following theorem.

\begin{theorem}\label{theorem anti-trivial}
Given integers $n,k>0$ and $m\geqslant 0$ and a matrix $M\in\M(n,m,k)$, the following conditions are equivalent:
\begin{enumerate}[label=(\arabic*)]
\item\label{anti-trivial 1} $M$ has $[0]_n$ among its columns.

\item\label{anti-trivial 2} $\mathbf{Set}$ has $M$-closed relations.

\item\label{anti-trivial 3} $M$ is anti-trivial.
\end{enumerate}
\end{theorem}

\begin{proof}
The implication \ref{anti-trivial 1}$\Rightarrow$\ref{anti-trivial 3} follows from the definition of $M$-closedness and the implication \ref{anti-trivial 3}$\Rightarrow$\ref{anti-trivial 2} is an immediate consequence of the definition of anti-trivial matrices. For the remaining implication \ref{anti-trivial 2}$\Rightarrow$\ref{anti-trivial 1}, we suppose that $\mathbf{Set}$ has $M$-closed relations. We consider the $n$-ary relation $R$ on $\{0,\dots,k-1\}$ formed by the columns of~$M$. Since this relation is $M$-closed, considering the interpretation of $M$ given by the identity function on $\{0,\dots,k-1\}$, we get that $[0]_n$ must belong to~$R$, i.e., to the columns of~$M$.
\end{proof}

Note that due to the Yoneda embedding, the equivalence \ref{anti-trivial 2}$\Leftrightarrow$\ref{anti-trivial 3} in the theorem above can be established directly without going through the condition~\ref{anti-trivial 1}.

\section{The algorithm}\label{sec3}

By a \emph{matrix set} we mean a subset $S$ of the union
$$\bigcup_{\substack{n > 0\\m \geqslant 0\\k>0}} \M(n,m,k).$$
We denote by $\mathsf{mclex}S$ the collection of all finitely complete categories which have $M$-closed relations for each matrix $M$ in~$S$. An \emph{$S$-injective object} in a category $\mathbb{C}$ is an object $X$ such that for every matrix $M\in S$ with any number $n$ of rows, every internal $n$-ary relation on any object in $\mathbb{C}^\mathsf{op}$ is $M$-closed over~$X$. If $\mathbb{C}$ is finitely cocomplete (i.e., if $\mathbb{C}^\mathsf{op}$ is finitely complete), this is equivalent to require that for every matrix $M\in S$ with any number $n$ of rows, every internal $n$-ary relation in $\mathbb{C}^\mathsf{op}$ is strictly $M$-closed over~$X$. The terminal object in a category, when it exists, is always $S$-injective. In fact, $S$-injective objects are closed under all limits that exist in the category. This follows from the fact that, given morphisms $(r_i\colon C\to R)_{1\leqslant i\leqslant n}$ in $\mathbb{C}$ representing a relation in $\mathbb{C}^\mathsf{op}$, the legs $(\pi_j\colon L \to X_j)_{j\in J}$ of a limit in~$\mathbb{C}$ and a family of morphisms $(f_i\colon C\to L)_{1\leqslant i\leqslant n}$, there exists a morphism $g\colon R\to L$ such that $gr_i=f_i$ for each $i\in\{1,\dots,n\}$ if and only if, for each $j\in J$, there exists a morphism $g_j\colon R\to X_j$ such that $g_jr_i=\pi_jf_i$ for each $i\in\{1,\dots,n\}$. We write $\mathsf{Inj}_S\mathbb{C}$ for the full subcategory of $\mathbb{C}$ consisting of all $S$-injective objects. We say that $\mathbb{C}$ \emph{has enough $S$-injective objects} if for any object $C$ in $\mathbb{C}$ there is a monomorphism $C\rightarrowtail D$ such that $D$ is $S$-injective.

For instance, as shown in~\cite{Weighill2017}, when $\mathbb{C}$ is the category of topological spaces and $S=\{M\}$, where $M$ is the Mal'tsev matrix (Example~\ref{ExaA}), $S$-injective objects are `$R_1$-spaces'~\cite{Davis1961}, i.e., topological spaces in which if one point can be separated from another by an open set, then the two points can be separated from each other by disjoint open sets. In contrast, every topological space is $S$-injective with $S=\{M\}$, when $M$ is the majority matrix (Example~\ref{ExaB}) (in other words, the dual of the category of topological spaces is a majority category~\cite{Hoefnagel2019a}). The category of topological spaces has thus enough $S$-injective objects for the two mentioned matrix sets~$S$. 

The following result is obtained by an adaptation of ideas from~\cite{Weighill2017}:

\begin{theorem}\label{ThmB}
Let $\mathbb{C}$ be a category having finite limits and finite colimits, where every morphism factorises as an epimorphism followed by an equaliser. Consider two matrix sets $S\subseteq T$. If $\mathbb{C}$ has enough $T$-injective objects, then $\mathsf{Inj}_S\mathbb{C}$ is the largest full subcategory of $\mathbb{C}$ among those that contain all $T$-injective objects, are closed under finite limits, and whose dual categories have $M$-closed relations for every $M\in S$.
\end{theorem}

\begin{proof}
A factorisation $f=gh$ of a morphism $f\colon X\to Y$ into an epimorphism $h$ followed by an equaliser $g$ is a universal epi-factorisation of~$f$. This is a consequence of the fact that any equaliser $g$ is a `strong monomorphism', meaning that for any commutative diagram of solid morphisms
$$\xymatrix{ W\ar[r]\ar@{->>}[d]_-{h'} & Z\ar@{ >->}[d]^-{g} \\ V\ar@{-->}[ur]^-{d}\ar[r] & Y }$$
where $h'$ is an epimorphism, there exists a unique dashed morphism $d$ keeping the diagram commutative. So, in view of Theorem~\ref{ThmA}, a matrix $M$ with $n$ rows is functional in an object $X$ if and only if every $n$-ary internal relation on any object in $\mathbb{C}^\mathsf{op}$ is $M$-closed over~$X$. As mentioned above, it is not difficult to see that $\mathsf{Inj}_S\mathbb{C}$ is closed under finite limits in~$\mathbb{C}$. Any object that is a $T$-injective object is an $S$-injective object, and so $\mathsf{Inj}_S\mathbb{C}$ contains all $T$-injective objects. To get convinced that $(\mathsf{Inj}_S\mathbb{C})^\mathsf{op}$ has $M$-closed relations for every $M\in S$, it suffices to check that any internal relation in $(\mathsf{Inj}_S\mathbb{C})^\mathsf{op}$ will remain an internal relation in the bigger category~$\mathbb{C}^\mathsf{op}$. So consider an internal relation $$r=\left[\begin{array}{c} r_1 \\ \vdots \\ r_n \end{array}\right],\quad r_i\colon R\to C_i\textrm{ for }i\in\{1,\dots,n\},$$
in $(\mathsf{Inj}_S\mathbb{C})^\mathsf{op}$, where $n$ is any positive integer. Now, in the category~$\mathbb{C}$, where each $r_i$ has opposite direction, decompose $r'=gh$ the induced morphism $r'\colon C_1+\dotsc+C_n\to R$ in $\mathbb{C}$ as an epimorphism $h$ followed by an equaliser~$g$. Let $f_1,f_2$ be the two morphisms of which $g$ is an equaliser. Let $D$ denote their common codomain. Consider a monomorphism $e\colon D\rightarrowtail E$ such that $E$ is a $T$-injective object. Then $g$ is still an equaliser of $ef_1,ef_2$, and since both $R$ and $E$ belong to $\mathsf{Inj}_S\mathbb{C}$, so does the domain of~$g$, and hence $g$ itself. Moreover, $g$ is an equaliser in~$\mathsf{Inj}_S\mathbb{C}$. On the other hand, since each $r_i$ factorises through~$g$, we get that $g$ is an epimorphism in $\mathsf{Inj}_S\mathbb{C}$. Being an epimorphism and an equaliser, it must be an isomorphism. Then $r'$ is an epimorphism and so $r$ is an internal relation in~$\mathbb{C}^\mathsf{op}$. This means that $(\mathsf{Inj}_S\mathbb{C})^\mathsf{op}$ has $M$-closed relations for every $M\in S$. It remains to show that any full subcategory of $\mathbb{C}$ containing all $T$-injective objects, closed under finite limits in $\mathbb{C}$ and the dual category having $M$-closed relations for each $M\in S$, is part of $\mathsf{Inj}_S\mathbb{C}$. Let $X$ be an object in such a subcategory~$\mathbb{D}$. Consider a matrix $M$ in $S\cap\M(n,m,k)$. We want to prove that $M$ is functional in~$X$. Consider the diagram
$$\xymatrix{
X & &\\ & & n X^k\ar@{->>}[ld]_-{r^X_M}\ar[rd]^-{\pi^X_M}\ar[llu]_-{\pi_{[0]_n}^X} & \\
& R_M^X\ar@{-->}[luu]^-{p^X_M}\ar@{ >->}[rr]_-{i^X_M} & & X^m
}$$
of solid morphisms in~$\mathbb{C}$ (as in Theorem~\ref{ThmA}). Since $\mathbb{D}$ is closed under finite limits, the objects $X^k$ and $X^m$ are in~$\mathbb{D}$, and moreover, $X^m$ is an $m$-fold product of $X$ in~$\mathbb{D}$. Furthermore, since $i_M^X$ is an equaliser, similar to the earlier argument with monomorphisms into $T$-injective objects, it will be an equaliser in~$\mathbb{D}$. Composing $r_M^X$ with sum (coproduct) injections $\iota_i\colon X^k\to nX^k$, we get an internal $n$-ary relation on $X^k$ in $\mathbb{C}^\mathsf{op}$ and hence in~$\mathbb{D}^\mathsf{op}$. Since $\mathbb{D}^\mathsf{op}$ has $M$-closed relations, there will be a morphism $p_M^X$ such that $p_M^Xr_M^X\iota_i=\pi_1$ for each $i\in\{1,\dots,n\}$. This then gives commutativity of the left triangle in the diagram above. Therefore, $M$ is functional in~$X$. The proof is now complete.         
\end{proof}

The following is an adaptation of ideas from Section~3 of~\cite{Hoefnagel2019a}. For each non-zero natural number~$n$, we consider the category $\mathbf{Rel}_n$ whose objects are pairs $(X,R)$ where $X$ is a set and $R$ is an $n$-ary relation on~$X$, and morphisms $(X,R)\to (X',R')$ are relation-preserving functions, i.e., functions $f\colon X\to X'$ such that there exists a (necessarily unique) dashed morphism rendering the diagram
$$\xymatrix{R \ar@{}[d]|(.21){}="A" \ar@{^{(}->}@<-2pt>"A";[d] \ar@{-->}[r] & R' \ar@{}[d]|(.21){}="B" \ar@{^{(}->}@<-2pt>"B";[d] \\ X^n \ar[r]_-{f^n} & X'^{\,n}}$$
commutative. The forgetful functor $\mathbf{Rel}_n\to\mathbf{Set}$ is a topological functor (see e.g.~\cite{Borceux1994b}) and therefore $\mathbf{Rel}_n$ is a complete and cocomplete category. Moreover, each morphism $f\colon (X,R)\to (X',R')$ in $\mathbf{Rel}_n$ factors as
$$\xymatrix{(X,R) \ar@{->>}[r]^-{f'} & (\mathsf{Im}\, f,(\mathsf{Im}\, f)^n\cap R') \ar@{ >->}[r]^-{i} & (X',R')}$$
where $\mathsf{Im}\, f$ is the image of the function~$f$, $i$ is the inclusion of this image in the codomain of~$f$ and $f'$ is the corestriction of~$f$. Since $f'$ is surjective it is an epimorphism in $\mathbf{Rel}_n$, and $i$ is the equaliser of
$$\xymatrix{(X',R') \ar@<3.5pt>[r]^-{g_1} \ar@<-3.5pt>[r]_-{g_2} & (Q,Q^n)}$$
where $g_1,g_2$ is the cokernel pair of $i$ in~$\mathbf{Set}$. Therefore, each morphism in $\mathbf{Rel}_n$ factors as an epimorphism followed by an equaliser and, in particular, $\mathbf{Rel}_n$ has universal epi-factorisations.

\begin{lemma}\label{LemB}
Given integers $n,n',k>0$ and $m\geqslant 0$, an object $(X,R)$ of $\mathbf{Rel}_n$ and a matrix $M\in\M(n',m,k)$, the following statements are equivalent:
\begin{enumerate}[label=(\arabic*)]
\item\label{Lem B 1} $(X,R)$ is an $\{M\}$-injective object in $\mathbf{Rel}_n$.
\item\label{Lem B 2} $M$ is functional in $(X,R)$.
\item\label{Lem B 3} $M$ is functional in the set $X$ and $R$ is $M$-sharp.
\end{enumerate}
Moreover, if $M$ is a non-trivial matrix, these are further equivalent to:
\begin{enumerate}[label=(\arabic*), resume]
\item\label{Lem B 4} $R$ is $M$-sharp.
\end{enumerate}
\end{lemma}

\begin{proof}
The equivalence \ref{Lem B 1}$\Leftrightarrow$\ref{Lem B 2} is a direct application of Theorem~\ref{ThmA}. To make explicit what \ref{Lem B 2} means, let us consider the following commutative diagram of plain morphisms in $\mathbf{Rel}_n$,
$$\xymatrix{
(X,R) & &\\ & & (n' X^k,n' R^k)\ar@{->>}[ld]_-{r^X_M}\ar[rd]^-{\pi^X_M}\ar[llu]_-{\pi_{[0]_{n'}}^X} & \\
& (\mathsf{Im}\, \pi^X_M,(\mathsf{Im}\, \pi^X_M)^n\cap R^m)\ar@{-->}[luu]^-{p^X_M}\ar@{ >->}[rr]_-{i^X_M} & & (X^m,R^m)
}$$
where $R^m$ is seen as a subset of $(X^m)^n$ via the canonical injection $R^m \rightarrowtail (X^n)^m \cong (X^m)^n$ and $n' R^k$ is seen as a subset of $(n' X^k)^n$ via the canonical injection $n' R^k \rightarrowtail n' (X^n)^k \cong n' (X^k)^n \rightarrowtail (n' X^k)^n$. Then, \ref{Lem B 2} means that there is a morphism $p^X_M$ in $\mathbf{Rel}_n$ retaining the commutativity of the diagram. The existence of a function $p^X_M$ retaining the commutativity of the diagram is exactly the definition of $M$ being functional in the set~$X$. For this function to be a morphism in $\mathbf{Rel}_n$, it means that for each $n$-tuple $(x_1,\dots,x_n)$ of elements in $n'X^k$ such that $(r^X_M(x_1),\dots,r^X_M(x_n))$ is in~$R^m$, then $(\pi_{[0]_{n'}}^X(x_1),\dots,\pi_{[0]_{n'}}^X(x_n))$ is in~$R$. Denoting the canonical injections $X^k\to n'X^k$ by $\iota_1,\dots,\iota_{n'}$, this condition can be equivalently written as: for each $(i_1,\dots,i_n)\in\{1,\dots,n'\}^n$ and each $(x_1,\dots,x_n)\in (X^k)^n$ such that $(\pi^X_M(\iota_{i_1}(x_1)),\dots,\pi^X_M(\iota_{i_n}(x_n)))$ is in~$R^m$, then $(\pi_{[0]_{n'}}^X(\iota_{i_1}(x_1)),\dots,\pi_{[0]_{n'}}^X(\iota_{i_n}(x_n)))$ is in~$R$. Denoting the entries of $M$ as in $M=[a_{ij}]_{i,j}$, we can further reformulate this condition as follows: for each $(i_1,\dots,i_n)\in\{1,\dots,n'\}^n$ and each matrix
$$\left[\begin{array}{ccc} x_{10} & \dots & x_{1\, k-1} \\ \vdots & & \vdots \\ x_{n0} & \dots & x_{n\, k-1} \end{array}\right]$$
of elements of~$X$, if the columns of the matrix
$$\left[\begin{array}{ccc} x_{1a_{i_11}} & \dots & x_{1a_{i_1m}} \\ \vdots & & \vdots \\ x_{na_{i_n1}} & \dots & x_{na_{i_nm}} \end{array}\right]$$
are elements of $R$ then so is
$$\left[\begin{array}{c} x_{10} \\ \vdots \\ x_{n0} \end{array}\right].$$
This condition is exactly expressing that $R$ is $M$-sharp, proving the equivalence \ref{Lem B 2}$\Leftrightarrow$\ref{Lem B 3}. Finally, if $M$ is not a trivial matrix, $M$ is always functional in $X$ by Theorem~\ref{ThmC}, proving the equivalence \ref{Lem B 3}$\Leftrightarrow$\ref{Lem B 4} in that case.
\end{proof}

We can then see that for a non-negative integer $n$ and a matrix set~$S$, the category $\mathsf{Inj}_S\mathbf{Rel}_n$ forms a (full) reflective subcategory of $\mathbf{Rel}_n$. Given an object $(X,R)$ in $\mathbf{Rel}_n$, its reflection $f\colon (X,R)\to (X',R')$ in the subcategory $\mathsf{Inj}_S\mathbf{Rel}_n$ is obtained as follows:
\begin{itemize}
\item If $S$ contains a trivial matrix, then $X'=X$ if $X$ is the empty set and $X'$ is a singleton if $X$ is non-empty. The function $f$ is the unique function $f\colon X\to X'$.

\item If $S$ does not contain a trivial matrix, then $X'=X$. In this case, $f$ is the identity function $f=1_X$.
\end{itemize}
We then define $R'\subseteq X'^{\, n}$ as the intersection of all relations containing $f^n(R)$ as a subrelation and which are $M$-sharp for each matrix $M$ in~$S$. Therefore, a colimit of a small diagram in $\mathsf{Inj}_S\mathbf{Rel}_n$ can be obtained by applying the construction above to the colimit of the same diagram in $\mathbf{Rel}_n$. In particular, each $(\mathsf{Inj}_S\mathbf{Rel}_n)^\mathsf{op}$ is a finitely complete category. These remarks together with Theorem~\ref{ThmB} and Lemma~\ref{LemB} bring us to the following result:

\begin{theorem}\label{ThmD}
Consider two matrix sets $S$ and~$U$. If for any non-zero natural number~$n$, every $n$-ary relation $R$ on any set $X$ that is $M$-sharp for every matrix $M$ in $S$ is $N$-closed for every matrix $N$ in $U$ having $n$ rows, then every finitely complete category that has $M$-closed relations for every $M$ in $S$ also has $N$-closed relations for every $N$ in~$U$, i.e., $\mathsf{mclex}S\subseteq\mathsf{mclex}U$. The converse is also true when no matrix in $S$ is trivial.  
\end{theorem}

\begin{proof}
Assume first for any integer $n>0$, every $n$-ary relation $R$ on any set $X$ that is $M$-sharp for every matrix $M$ in $S$ is $N$-closed for every matrix $N$ in $U$ with $n$ rows. Consider a finitely complete category $\mathbb{C}$ that has $M$-closed relations for every $M$ in~$S$. Then every internal relation in $\mathbb{C}$ is $M$-sharp for each $M\in S$. Considering a matrix $N\in U$ with $n$ rows, an $n$-ary internal relation $r\colon R \rightarrowtail C^n$ on an object $C$ in~$\mathbb{C}$ and an arbitrary object~$Y$, we need to prove that $r$ is $N$-closed over~$Y$. These induce an $n$-ary relation $R'$ on the set $\mathbb{C}(Y,C)$ for which $g_1,\dots,g_n\colon Y\to C$ are related if and only if there exists a (necessarily unique) morphism $v\colon Y \to R$ such that $r_iv=g_i$ for each $i\in\{1,\dots,n\}$. It is then easy to see that $R'$ is an (ordinary) $n$-ary relation which is $M$-sharp for each $M\in S$. Our assumption gives then that $R'$ is $N$-closed, which implies that the internal relation $r$ on $C$ is $N$-closed over~$Y$.

Now assume that every finitely complete category that has $M$-closed relations for every $M$ in $S$ also has $N$-closed relations for every $N$ in~$U$. Suppose no matrix in $S$ is trivial. Then no matrix in $U$ can be trivial as well. Indeed, if $U$ contained a trivial matrix, then a non-preorder could not belong to $\mathsf{mclex}U$, but at the same time $\mathbf{Set}^{\mathsf{op}}$ belongs to $\mathsf{mclex}S$ by Theorem~\ref{ThmC}. For each positive~$n$, consider the full subcategory $\mathsf{Inj}_U\mathbf{Rel}_n$ of $\mathbf{Rel}_n$ consisting of $U$-injective objects. We will now apply Theorem~\ref{ThmB} to the category $\mathbb{C}=\mathbf{Rel}_n$. Let $T$ be the set of all non-trivial matrices. For any $n>0$ and any set~$X$, the relation $X^n$ is clearly $M$-sharp for every matrix~$M$. So the object $(X,X^n)$ is a $T$-injective object in $\mathbf{Rel}_n$. For any object $(X,R)$ in $\mathbf{Rel}_n$ there is a monomorphism into such object --- namely, the inclusion $(X,R)\rightarrowtail (X,X^n)$. By Theorem~\ref{ThmB} we then get that $\mathsf{Inj}_U\mathbf{Rel}_n$ is the largest full subcategory of $\mathbf{Rel}_n$ having the following properties: it contains all $T$-injective objects, it is closed under finite limits and its dual category has $N$-closed relations for every $N$ in~$U$. By Theorem~\ref{ThmB} again, the subcategory $\mathsf{Inj}_S\mathbf{Rel}_n$ of $\mathbf{Rel}_n$ consisting of all $S$-injective objects has similar properties: it contains all $T$-injective objects, it is closed under finite limits and its dual category has $M$-closed relations for every $M$ in~$S$. The assumption that every finitely complete category having $M$-closed relations for every $M$ in $S$ has $N$-closed relations for every $N$ in $U$ gives that $\mathsf{Inj}_S\mathbf{Rel}_n$ must be a subcategory of $\mathsf{Inj}_U\mathbf{Rel}_n$ (note that, from the paragraph before the theorem, we know that the duals of these categories are finitely complete). So, by Lemma~\ref{LemB}, for any non-zero natural number~$n$, every $n$-ary relation $R$ on any set $X$ that is $M$-sharp for every matrix $M$ in $S$ is also $N$-sharp for every matrix $N$ in~$U$. In particular, this means that every such relation is $N$-closed for every matrix $N$ in $U$ having $n$ rows.
\end{proof}

Note that, if $S$ contains only the trivial matrix
$$M=\left[\begin{array}{cc|c} 1 & 0 & 0\\ 0 & 1 & 0\end{array}\right]$$
and $U$ contains only the trivial matrix
$$N=\left[\begin{array}{c|c} 1 & 0 \\ 1 & 0\end{array}\right],$$
both $\mathsf{mclex}S$ and $\mathsf{mclex}U$ are the matrix class of all finitely complete preorders. Now, given two subsets $R_1\subseteq X$ and $R_2\subseteq X$ of a set~$X$, the binary relation $R_1\times R_2 \subseteq X^2$ is $M$-sharp; however, this relation is $N$-closed only if $R_1$ and $R_2$ are disjoint or if $R_1=R_2=X$. This shows that the converse implication in Theorem~\ref{ThmD} might not hold if $S$ contains a trivial matrix.

From the proof of Theorem~\ref{ThmD} we can extract the following theorem, which says that in order to prove an implication between (non-empty) matrix properties (in the finitely complete context), one only needs to produce a proof for a single particular category.

\begin{theorem}\label{ThmF}
Let $S$ be a matrix set which contains no empty matrix and let $N\in\M(n,m,k)$, where $n,k>0$ and $m \geqslant 0$. The following statements are equivalent:
\begin{enumerate}[label=(\arabic*)]
\item $\mathsf{mclex}S\subseteq\mathsf{mclex}\{N\}$.
\item The category $(\mathsf{Inj}_S\mathbf{Rel}_n)^\mathsf{op}$, which belongs to $\mathsf{mclex}S$, also belongs to $\mathsf{mclex}\{N\}$.
\end{enumerate}
\end{theorem}

(Notice that in the case where $S$ contains a non-empty trivial matrix, both statements are equivalent to $N$ being a non-empty matrix. The result thus also holds in that case.)

We now come to the following question. Given a matrix set $S$ and a matrix $N\in\M(n,m,k)$, how to decide whether every $n$-ary relation $R$ on any set $X$ that is $M$-sharp for every matrix $M$ in $S$ is $N$-closed? Consider the $n$-ary relation on $\{0,\dots,k-1\}$ given by the columns of~$N$. Let us write $\mathsf{col}(N)$ for this relation. Consider the intersection of all $n$-ary relations on $\{0,\dots,k-1\}$ that contain $\mathsf{col}(N)$ as a subrelation and that are $M$-sharp for all $M$ in~$S$. Let us denote this by $\mathsf{col}^S(N)$. It is certainly $M$-sharp for all $M$ in~$S$, and so, if every $n$-ary relation on any set $X$ that is $M$-sharp for every matrix $M$ in $S$ is $N$-closed, then $[0]_n\in \mathsf{col}^S(N)$. The converse is also true: If $[0]_n\in \mathsf{col}^S(N)$, then every $n$-ary relation $R$ on any set $X$ that is $M$-sharp for every matrix $M$ in $S$ is $N$-closed. To see this, consider an interpretation 
$$\left[\begin{array}{ccc|c} x_{11} & \dots & x_{1m} & y_1\\ \vdots & & \vdots & \vdots\\ x_{n1} & \dots & x_{nm} & y_n \end{array}\right]$$
of $N$ of type $X$ such that every left column of the interpretation belongs to~$R$. Let this interpretation be given by a map $f\colon \{0,\dots,k-1\}\to X$. Then the inverse image $f^{-1}(R)$ of $R$ along $f$ will be a relation on $\{0,\dots,k-1\}$ that is $M$-sharp for every $M$ in~$S$. Indeed, every reduction of each $M$ of type $(\{0,\dots,k-1\},\dots,\{0,\dots,k-1\})$ whose (left) columns belong to $f^{-1}(R)$ will have a further reduction by applying $f$ to its entries, whose (left) columns belong to~$R$. Then, since $R$ is $M$-sharp, $f$ of the right column of the reduction will be in~$R$, which means the right column of the reduction will be in $f^{-1}(R)$. Now, since $f^{-1}(R)$ is $M$-sharp for every $M$ in $S$ and since it contains $\mathsf{col}(N)$, it must also contain $\mathsf{col}^S(N)$. The fact that $[0]_n\in \mathsf{col}^S(N)$ will now give that the right column of the above interpretation of $N$ belongs to~$R$. So Theorem~\ref{ThmD} has the following consequence.

\begin{corollary}\label{CorA}  
Consider two matrix sets $S$ and~$U$. If $[0]_n\in\mathsf{col}^S(N)$ for every non-zero natural number $n$ and every matrix $N$ in $U$ with $n$ rows, then every finitely complete category that has $M$-closed relations for every $M$ in $S$ also has $N$-closed relations for every $N$ in~$U$, i.e., $\mathsf{mclex}S\subseteq\mathsf{mclex}U$. The converse is also true when no matrix in $S$ is trivial.
\end{corollary}

Note that $\mathsf{col}^S(N)$ is necessarily finite (it is a subset of $\{0,\dots,k-1\}^n$, when $N\in\M(n,m,k)$). When $S$ is finite, we can build $\mathsf{col}^S(N)$ in finitely many steps as follows. For a given $R\subseteq \{0,\dots,k-1\}^n$, we consider the relation $S(R)\subseteq \{0,\dots,k-1\}^n$ containing exactly $R$ and the right columns of each row-wise interpretation $B$ of type $(\{0,\dots,{k-1}\},\dots,\{0,\dots,k-1\})$ of each matrix $M'\in\M(n,m',k')$ whose rows are rows of a common matrix $M$ in $S$ such that the (left) columns of $B$ belong to~$R$. The number of all possible such $B$'s is finite (as $S$ is finite), so $S(R)$ can be built from $R$ in finitely many steps. It is easy to see that if $R\subseteq \mathsf{col}^S(N)$ then also $S(R)\subseteq \mathsf{col}^S(N)$. Now, starting with $R=\mathsf{col}(N)$ we can build a chain of proper subset inclusions
$$\mathsf{col}(N)\subset S(\mathsf{col}(N))\subset SS(\mathsf{col}(N))\subset\dots\subset SS\dots S(\mathsf{col}(N))$$
until $S(R)$ for the last set $R$ in the chain is equal to~$R$. When this happens, $R$ will be $M$-sharp for every $M$ in~$S$, and hence $R=\mathsf{col}^S(N)$. The chain does terminate after some finitely many steps because each of the sets in the chain are subsets of the finite set $\{0,\dots,k-1\}^n$. 

From the results above we can readily extract an algorithm for deciding $\mathsf{mclex}S\subseteq \mathsf{mclex}U$, when $S$ and $U$ are finite matrix sets. Note that the results did not give full characterisation of $\mathsf{mclex}S\subseteq \mathsf{mclex}U$ when $S$ contains a trivial matrix. This case can be dealt with separately as follows:
\begin{itemize}
\item If $S$ contains a trivial matrix, then every category in $\mathsf{mclex}S$ is a preorder. Therefore, one has $\mathsf{mclex}S\nsubseteq \mathsf{mclex}U$ if and only if $U$ contains an empty matrix and $S$ does not contain an empty matrix.
\end{itemize}
The combined algorithm which deals with both the trivial and non-trivial matrices is then the following. To decide whether $\mathsf{mclex}S\subseteq \mathsf{mclex}U$ where $S$ and $U$ are finite matrix sets, do the following.
\begin{description}
	\item[Step 0 (dealing with empty matrices).] If $S$ contains an empty matrix, then terminate the process with positive decision for $\mathsf{mclex}S\subseteq \mathsf{mclex}U$. Else, if $U$ contains an empty matrix then terminate the process with negative decision for $\mathsf{mclex}S\subseteq \mathsf{mclex}U$.

	\item[Step 1 (dealing with trivial matrices).] If $S$ contains a matrix $M$ which contains a row with no $0$'s then terminate the process with positive decision for $\mathsf{mclex}S\subseteq \mathsf{mclex}U$. Else, for each matrix $M$ in~$S$, for each pair of distinct rows in~$M$ and for a (randomly) chosen entry $0$ in the first row, do the following. If the second row does not contain a $0$ entry that admits a linkage with the chosen $0$ entry in the first row, then terminate the process with positive decision for $\mathsf{mclex}S\subseteq \mathsf{mclex}U$.

    \item[Step 2 (dealing with non-trivial matrices).] For each matrix $N\in\M(n,m,k)$ in $U$ do the following. Keep expanding~$N$, until it is impossible to expand it further, with right columns of those row-wise interpretations $B$ of type $(\{0,\dots,{k-1}\},\dots,$ $\{0,\dots,{k-1}\})$ of each matrix $M'\in\M(n,m',k')$ whose rows are rows of a common matrix $M$ in $S$ such that the left columns of~$B$, but not its right column, can be found in~$N$. If the expanded $N$ does not contain the column of~$0$'s, then terminate the process with negative decision for $\mathsf{mclex}S\subseteq \mathsf{mclex}U$.

    \item[Step 3 (conclusion).] Reaching this step means that the process has not been terminated in the previous steps. Then the process completes with positive decision for $\mathsf{mclex}S\subseteq \mathsf{mclex}U$.
\end{description}

An obvious question that arises here is whether a conjunction of matrix properties is a matrix property. It turns out that if $S=\{M_1,\dots,M_s\}$ is a finite matrix set, then it is possible to create a matrix $M_1\times\cdots\times M_s$ such that $\mathsf{mclex}S=\mathsf{mclex}\{M_1\times\cdots\times M_s\}$. If $S=\varnothing$, we set the empty product of matrices to be the matrix $[0]\in\M(1,1,1)$. If $S\neq \varnothing$ and if $M_i\in\M(n_i,m_i,k_i)$ for each $i\in\{1,\dots,s\}$, we define a product $M_1\times\cdots\times M_s$ of matrices to be a matrix in
$$\M(n_1+\dotsc+n_s,m_1\cdot \dotsc\cdot m_s,\mathsf{max}\{k_1,\dots,k_s\})$$
whose columns are indexed by the set of all possible choices of a column in each of $M_1,\dots,M_s$ and are obtained by stacking each of these columns on top of each other in the same order as the matrices in the list. Thus,  
$$\mathsf{col}(M_1\times\cdots\times M_s)\cong\mathsf{col}(M_1)\times \cdots \times \mathsf{col}(M_s).$$
It is easy to see that if $m_1\cdot \dotsc\cdot m_s>0$, then each $M_i\in S$ will be a reduction of such~$M_1\times\cdots\times M_s$ --- just keep those rows of $M_1\times\cdots\times M_s$ which correspond to the position of columns from $M_i$ in the stack and then delete the duplicate columns if necessary. So $\mathsf{mclex}{\{M_1\times\cdots\times M_s\}}\subseteq \mathsf{mclex}S$ by Lemma~\ref{Lemma reductions}. In the case where $m_1\cdot \dotsc\cdot m_s=0$ or where $S=\varnothing$, this inclusion holds trivially. To get the converse inclusion, it is sufficient to argue in the case when $s=2$, since the cases $s=0$ and $s=1$ are obvious and when $s\geqslant 3$, we have
$$M_1\times\cdots\times M_s=(\dots((M_1\times M_2)\times M_3)\dots)\times M_s.$$ 
So we assume $S=\{M_1,M_2\}$. For any fixed column $C$ of~$M_2$, the matrix $M_1\times\{C\}$ is a reduction of~$M_1$. The right column of this reduction is given by $\{[0]_{n_1}\}\times \{C\}$. The matrix formed by these right columns (one for each column $C$ of~$M_2$) is a reduction of~$M_2$, with the right column being $[0]_{n_1+n_2}$ and so $$\mathsf{mclex}\{M_1,M_2\}\subseteq \mathsf{mclex}\{M_1\times M_2\},$$
by Corollary~\ref{CorA}. We have thus proved the following:

\begin{theorem}\label{ThmE}
For a finite matrix set~$S$, we have:
$$\mathsf{mclex}S=\bigcap_{M\in S}\mathsf{mclex}\{M\}=\mathsf{mclex}\left\{\prod_{M\in S} M\right\}.$$ 
\end{theorem} 

This theorem together with Theorem~\ref{ThmC} gives:

\begin{corollary}
For a finite matrix set~$S$, the following conditions are equivalent:
\begin{enumerate}[label=(\arabic*)]
\item\label{mclexS trivial} $\mathsf{mclex}S$ only contains preorders (i.e., $\prod_{M\in S} M$ is a trivial matrix).
\item\label{S contains a trivial matrix} $S$ contains a trivial matrix.
\end{enumerate}
\end{corollary}

\begin{proof}
The implication \ref{S contains a trivial matrix}$\Rightarrow$\ref{mclexS trivial} being trivial, let us show \ref{mclexS trivial}$\Rightarrow$\ref{S contains a trivial matrix}. Assuming~\ref{mclexS trivial}, we know that $S$ cannot be empty. If $S$ contains only one matrix, the result is obvious. Since the intersection of finitely many matrix classes is again a matrix class (Theorem~\ref{ThmE}), by induction it suffices to treat the case where $S=\{M_1,M_2\}$ contains exactly two matrices. We thus suppose that $M_1\times M_2$ is a trivial matrix and we shall prove that either $M_1$ or $M_2$ is trivial. Without loss of generality, we can assume that both $M_1$ and $M_2$ are non-empty (and thus so is $M_1\times M_2$). If a row of $M_1\times M_2$ does not contain $0$ as an entry, so does the corresponding row in $M_1$ or $M_2$ and thus at least one of $M_1$ and $M_2$ is trivial by Theorem~\ref{ThmC}. Again by Theorem~\ref{ThmC}, the remaining possibility is that $M_1\times M_2$ admits two distinct rows containing $0$ as an entry but for which the two corresponding entries cannot be linked by a linkage. It is not possible that one of these rows comes from $M_1$ and the other one from $M_2$ since in that case, by construction of $M_1\times M_2$, there would exist a column of $M_1\times M_2$ with $0$ in each of these rows. So these two rows must come from the same matrix, which is then forced to be a trivial matrix.
\end{proof}

\section{Computer-aided classification results}\label{sec4}

We write $\mathsf{Mclex}$ for the collection of all matrix classes (of finitely complete categories). It is a meet semi-lattice by Theorem~\ref{ThmE}. For each $n>0$, let $D_n\in\M(n,n,2)$ denote the matrix, all of whose entries are $0$ except those on the main diagonal, which are~$1$. Note that $D_1$ and $D_2$ are the matrices from Theorem~\ref{ThmC} characterising the trivial matrices, so
$$\mathsf{mclex}\{D_1\}=\mathsf{mclex}\{D_2\}$$
is the matrix class of finitely complete preorders. $D_3$ is the matrix that defines majority categories (Example~\ref{ExaB}), and so we have a strict inclusion $\mathsf{mclex}\{D_2\}\subsetneqq \mathsf{mclex}\{D_3\}$. Using the algorithm we can see that this extends to an infinite sequence of strict inclusions
$$\mathsf{mclex}\{D_2\}\subsetneqq \mathsf{mclex}\{D_3\}\subsetneqq \mathsf{mclex}\{D_4\}\subsetneqq\cdots.$$  
Hence $\mathsf{Mclex}$ is infinite. It is not difficult to see that $\mathsf{Mclex}$ is countable. Computer implementation of our algorithm allows one to get a complete description of the posets of matrix classes given by matrices with some restricted dimensions (dependent on the computational efficiency of the implementation and the computer). The computation of the poset of all matrix classes remains an open problem. The fragment of $\mathsf{Mclex}$ consisting of matrix classes given by all matrices in $\M(n,m,k)$ will be denoted by $\mathsf{Mclex}[n,m,k]$. Each $\mathsf{Mclex}[n,m,k]$ is obviously finite (it contains at most $k^{n\cdot m}$ different matrix classes). Our algorithm allows one to compute each poset $\mathsf{Mclex}[n,m,k]$ in a finite time that depends (exponentially) on the parameters $n,m,k$.

What can be established without the help of a computer, using just Lemma~\ref{LemA}, Theorem~\ref{ThmC}, Corollary~\ref{CorB} and Theorem~\ref{theorem anti-trivial}, is the following (where $n,m,k$ are integers such that $n,k>0$ and $m\geqslant 0$):
\begin{itemize}
\item $\mathsf{Mclex}[n,0,k]$ has exactly one matrix class consisting of the preorders with a single isomorphism class of objects (i.e., preorders equivalent to the terminal category).

\item When $m\geqslant 1$, we have that $\mathsf{Mclex}[n,m,1]$ has exactly one matrix class given by all finitely complete categories.

\item When $m\geqslant 1$ and $k\geqslant 2$, we have $\mathsf{Mclex}[1,m,k]=\mathsf{Mclex}[n,1,k]=\mathsf{Mclex}[n,2,k]$ and these have exactly two elements: the matrix class of finitely complete preorders and the matrix class of all finitely complete categories.

\item When $m\geqslant 3$ and $k\geqslant 2$, the set $\mathsf{Mclex}[2,m,k]$ consists of exactly three elements: the matrix class of finitely complete preorders, the matrix class of all finitely complete categories and the matrix class of Mal'tsev categories.
\end{itemize}
In view of this, we will call a matrix class \emph{degenerate} if it is determined by a trivial matrix or an anti-trivial matrix. That is, the degenerate matrix classes are the following ones:
\begin{itemize}
\item the matrix class consisting of the preorders with a single isomorphism class of objects (i.e., the bottom element of the poset $\mathsf{Mclex}$);
\item the matrix class of finitely complete preorders (i.e., the unique atom of the poset $\mathsf{Mclex}$);
\item the matrix class of all finitely complete categories (i.e., the top element of the poset $\mathsf{Mclex}$).
\end{itemize}

Let us notice that a matrix with entries in the set $\{0,\dots,k-1\}$ and having $n$ rows can have at most $k^n$ different columns. One of these $k^n$ columns is a column of zeros, whose presence makes the matrix anti-trivial. So all non-degenerate members of $\mathsf{Mclex}[n,m,k]$ lie in $\mathsf{Mclex}[n,k^n-1,k]$.

The computer program that we wrote at the time of preparing this paper is able to compute the posets $\mathsf{Mclex}[n,m,2]$, where $n\leqslant 4$, in a short amount of time. For two given matrices $M_1\in\M(n_1,m_1,k_1)$ and $M_2\in\M(n_2,m_2,k_2)$, establishing whether $\mathsf{mclex}\{M_1\}\subseteq\mathsf{mclex}\{M_2\}$ using our algorithm, is already quite laborious. In particular, in Step~2 of the algorithm we are required to do a column-comparison of $M_2$ with each row-wise interpretation $B$ of type $(\{0,\dots,k_2-1\},\dots,\{0,\dots,k_2-1\})$ of each matrix $M'_1\in\M(n_2,m_1,k_1)$ whose rows are rows of~$M_1$; and if the outcome is positive (i.e., if every column of $B$ is a column of~$M_2$) and if the right column of $B$ is not already in~$M_2$, then we need to do the same, again for all~$B$, after $M_2$ has been expanded with the right column of~$B$. In general, the number of these $M'_1$'s is $n_1^{n_2}$ and the number of such row-wise interpretations is $k_2^{k_1n_2}$ for each~$M'_1$. The maximum number of times that $M_2$ may get expanded is given by $k_2^{n_2}-1-m_2$, where $k_2^{n_2}-1$ is the total number of different non-zero columns for a matrix in $\M(n_2,m_2,k_2)$. So altogether, the process may require up to $n_1^{n_2}k_2^{k_1n_2}(k_2^{n_2}-1-m_2)$ many column-comparisons of two matrices. Therefore, when computing the poset $\mathsf{Mclex}[n,m,k]$, we want to decrease the number of times we would have to decide some inclusion $\mathsf{mclex}\{M_1\}\subseteq\mathsf{mclex}\{M_2\}$. This can be done for instance by removing duplicate rows or columns from matrices in $\M(n,m,k)$ and thus consider the resulting matrices only once. Since, just like duplication of rows and columns, also the order in which rows and columns are arranged does not alter the corresponding matrix class, we may only consider matrices for which both the rows and the columns are ordered increasingly with respect to the usual lexicographical order. Adapting the terminology from~\cite{Lubiw1987}, we say that such matrices are \emph{doubly lexi-ordered}. Other filters may also be applied. For instance, in the case when $k>2$, we may additionally filter out those matrices where, in some row and for some $0<i<k-1$, no entry $i$ is found before the first occurrence of~$i+1$.

Before displaying the results of our computer-aided computation of $\mathsf{Mclex}[n,m,k]$ for various $n,m,k$, we describe how we chose to represent the matrix classes in the display. For each matrix class $\mathcal{C}$ obtained from a non-empty matrix, and for each integers $n,m,k>0$, consider the set $\mathcal{C}_{n,m,k}$ of all matrices $M$ such that $\mathcal{C}=\mathsf{mclex}\{M\}$ and $M\in \M(n',m',k')$ for some $n'\leqslant n$, $m'\leqslant m$ and $k'\leqslant k$. Now consider the subset $\mathcal{C}^{\mathsf{R}}_{n,m,k}$ of $\mathcal{C}_{n,m,k}$ consisting of those matrices that have minimal number of rows; the subset $\mathcal{C}^{\mathsf{RC}}_{n,m,k}$ of $\mathcal{C}^{\mathsf{R}}_{n,m,k}$ consisting of those matrices that have minimal number of columns and the subset $\mathcal{C}^{\mathsf{RCG}}_{n,m,k}$ of $\mathcal{C}^{\mathsf{RC}}_{n,m,k}$ consisting of those matrices that have minimal greatest entry. Viewing elements of $\mathcal{C}^{\mathsf{RCG}}_{n,m,k}$ as sequences of elements of $\{0,\dots,k-1\}$ by juxtaposing the transpose of each column next to each other, a matrix will be called an \emph{$(n,m,k)$-canonical} matrix if it is the smallest element of $\mathcal{C}^{\mathsf{RCG}}_{n,m,k}$ by the lexicographical ordering. It is easy to see that any non-empty matrix $M\in\M(n',m',k')$ which is $(n,m,k)$-canonical for some $n\geqslant n'$, $m \geqslant m'$ and $k \geqslant k'$ is also $(n',m',k')$-canonical. However, as we will show later on in this section, there is a matrix in $\M(4,5,2)$ which is $(4,5,2)$-canonical but not $(4,5,3)$-canonical.

\begin{lemma}
Let $n,m,k\geqslant 1$ be integers and $M$ an $(n,m,k)$-canonical matrix. The following properties hold for~$M$:
\begin{itemize}
\item $M$ has no duplicate rows or columns.
\item $M$ is doubly lexi-ordered.
\item For each row of $M$ and for each $0<i<j<k$, if $j$ appears in the row, then so does $i$ and the first occurrence of $j$ appears after the first occurrence of~$i$.
\end{itemize}
\end{lemma}

\begin{proof}
The fact that $M$ does not have duplicate rows or columns follows immediately from the fact that removing a duplicate row or a duplicate column in a matrix does not alter the corresponding matrix class.

To prove that $M$ is doubly lexi-ordered, we use the fact that rearranging rows of a matrix, or rearranging its columns, does not alter the corresponding matrix class. The fact that the columns of $M$ are lexicographically ordered is easy to see. Indeed, if it is not so, at the first position where this breaks, swap the two consecutive columns. Then, the new sequence of juxtaposed transposed columns will be lexicographically less than the previous one, violating our assumption that $M$ is $(n,m,k)$-canonical.

Suppose now that the rows of $M$ are not lexicographically ordered, and let $i$ be the position of the first row whose consecutive row is lexicographically less than it. Let $j$ be the first position of an entry $a_{i,j}$ in the $i$-th row which is bigger than the corresponding entry $a_{i+1,j}$ in the $(i+1)$-th row. Swap the $i$-th and the $(i+1)$-th rows. Comparing juxtaposed arrangement of transposed columns for the old and the new matrices, we note that the first position where the two sequences will be different from each other is where the first sequence has entry $a_{i,j}$ and the second sequence has entry $a_{i+1,j}$. The second sequence is then lexicographically smaller than the first, which contradicts the assumption that $M$ is $(n,m,k)$-canonical.

The last statement follows from the fact that, given positive integers $i$ and~$j$, if one replaces in a row of a matrix, each entry $i$ by $j$ and vice-versa, the corresponding matrix class remains the same.
\end{proof}

We will now display some of the posets $\mathsf{Mclex}[n,m,k]$, where $n\in\{3,4\}$ and $k=2$, computed by a computer-implementation of the algorithm from the previous section. In these displays, we will represent each element of $\mathsf{Mclex}[n,m,k]$ by an image which represents the $(n,m,k)$-canonical matrix for that element of $\mathsf{Mclex}[n,m,k]$. Each white square of that image represents an entry~$0$ in the $(n,m,k)$-canonical matrix and each grey square represents an entry~$1$. In each display, an arrow connecting a matrix $M_1$ with a matrix $M_2$ marks inclusion of matrix classes $\mathsf{mclex}\{M_1\}\subseteq\mathsf{mclex}\{M_2\}$. We do not draw an arrow if it can be obtained by a path of arrows (we thus display the reflexive and transitive reduction of the poset). We also omit the degenerate matrix classes in these displays.

We begin with Figure~\ref{fig:3xmx2}, which displays the poset of non-degenerate members of $\mathsf{Mclex}[3,7,2]$. Since $2^3-1=7$, we know that Figure~\ref{fig:3xmx2} displays the poset of all non-degenerate matrix classes defined by a matrix having three rows, an arbitrary number of columns and whose entries lie in the set $\{0,1\}$. This figure also illustrates Corollary~\ref{CorB}: there are no matrices in the figure with a single row, and there is only one that has exactly $2$ rows --- namely, the matrix
$$M=\left[\begin{array}{ccc} 0 & 1 & 1\\ 1 & 0 & 1 \end{array}\right]$$ which, up to permutation of columns, is the Mal'tsev matrix (Example~\ref{ExaA}).

\begin{figure}
    \centering
    \includegraphics[width=130pt]{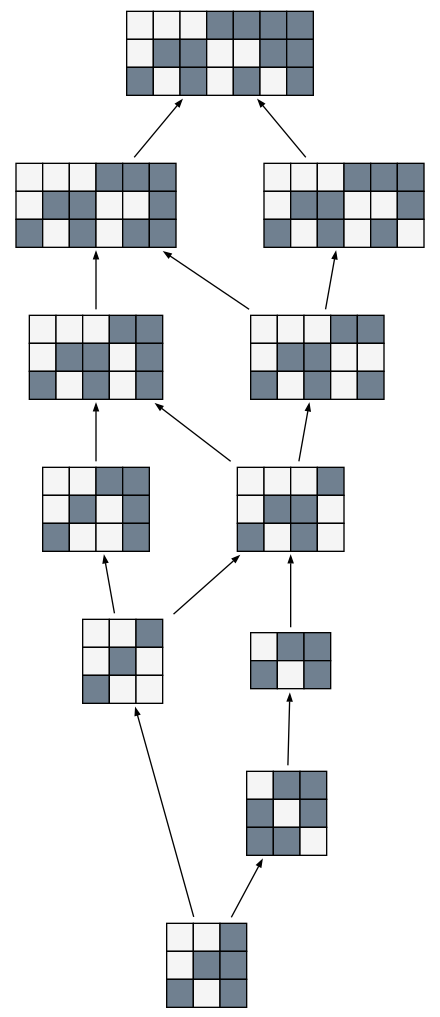}
    \caption{Hasse diagram of the poset of non-degenerate matrix classes in $\mathsf{Mclex}[3,7,2]$.}
    \label{fig:3xmx2}
\end{figure}

The matrix class on the left of the matrix class of Mal'tsev categories in Figure~\ref{fig:3xmx2} is the matrix class of majority categories (Example~\ref{ExaB}). The bottom matrix class is the finitely complete extension of the collection of arithmetical categories (Example~\ref{ExaC}) --- this is in fact the intersection of the matrix classes of Mal'tsev and majority categories. The one between arithmetical and Mal'tsev matrix classes is the one whose algebraic members (i.e., categories of algebras in a variety) are given by varieties having a so-called `minority term'.

What we cannot find in Figure~\ref{fig:3xmx2} is the matrix class arising from the syntactical refinement of the Mal'tsev condition characterising Mal'tsev varieties with directly decomposable congruences, from Remark~\ref{RemB}:
$$M=\left[\begin{array}{ccccc} 0 & 0 & 0 & 1 & 1 \\
1 & 0 & 1 & 0 & 1 \\
1 & 0 & 0 & 1 & 1 \\
1 & 1 & 0 & 1 & 0
\end{array}\right].$$
It turns out that the matrix class for the $M$ above shows itself only in Figure~\ref{fig:4x4x2}. The matrix class in Figure~\ref{fig:4x4x2} determined by the matrix $M$ above is the middle one in the fourth row (from the top). The matrix displaying that matrix class in the figure is given by
$$N=\left[\begin{array}{cccc} 
0 & 0 & 0 & 1 \\
0 & 0 & 1 & 1 \\
0 & 1 & 1 & 0 \\
1 & 0 & 1 & 1 \end{array}\right].$$
It is not difficult to see that $M$ has the following doubly lexi-ordering (shift column $1$ right to the position of column $3$ and interchange rows $2$ and~$3$):
$$M'=\left[\begin{array}{ccccc}
0 & 0 & 0 & 1 & 1 \\
0 & 0 & 1 & 1 & 1 \\
0 & 1 & 1 & 0 & 1 \\
1 & 0 & 1 & 1 & 0
\end{array}\right].$$
We can then see that $N$ is just $M'$ with the last column removed. So $\mathsf{mclex}\{N\}\subseteq \mathsf{mclex}\{M\}$. With the help of the computer we could find the following proof for the converse inclusion. The display below summarises the proof. It consists of three blocks of extended matrices (each having four rows). The first block has its left columns given by the matrix $M'$ and its right columns given by the matrix $N$. In every next block, each row is a row-wise interpretation of type $(\{0,1\},\dots,\{0,1\})$ of a row of $M'$ and every left column is a column appearing as a right column in one of the previous blocks. Reaching a right column of $0$'s confirms $\mathsf{mclex}\{M\}\subseteq\mathsf{mclex}\{N\}$. 
$$\begin{array}{ccccc|cccc} 
0 & 0 & 0 & 1 & 1 & 0 & 0 & 0 & 1\\
0 & 0 & 1 & 1 & 1 & 0 & 0 & 1 & 1\\
0 & 1 & 1 & 0 & 1 & 0 & 1 & 1 & 0\\
1 & 0 & 1 & 1 & 0 & 1 & 0 & 1 & 1\\
\hline
1 & 0 & 0 & 1 & 0 & 1 \\
1 & 0 & 0 & 1 & 0 & 1 \\
0 & 1 & 0 & 0 & 1 & 1 \\
1 & 0 & 1 & 1 & 0 & 0 \\
\hline
0 & 0 & 0 & 1 & 1 & 0 \\
0 & 0 & 1 & 1 & 1 & 0 \\
0 & 1 & 1 & 0 & 1 & 0 \\
1 & 0 & 1 & 1 & 0 & 0 \\
\end{array}$$
We will call such proof displays `lex-tableaux' (`lex' is a standard abbreviation of `left exact'). A lex-tableau is of course nothing other than a visual implementation of our algorithm.

\begin{figure}
    \centering
    \includegraphics[width=420pt]{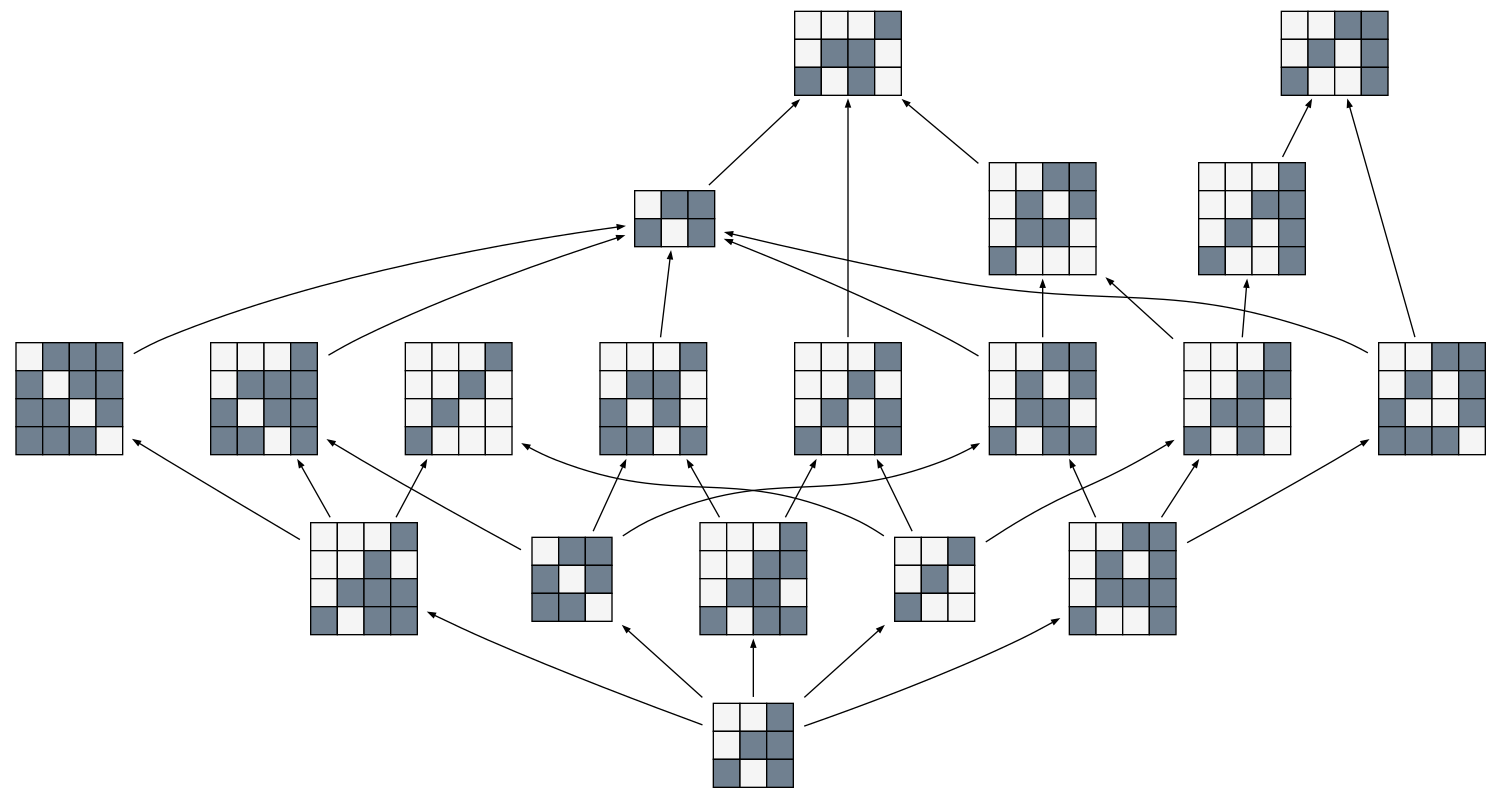}
    \caption{Hasse diagram of the poset of non-degenerate  matrix classes in $\mathsf{Mclex}[4,4,2]$.}
    \label{fig:4x4x2}
\end{figure}

What appears to be distinctive to the middle matrix in the fourth row of Figure~\ref{fig:4x4x2}, compared to the other matrices in that row, is that each of them has three arrows going out, while the former has two. The `missing arrow' gets added in Figure~\ref{fig:4x5x2} from the Introduction, which displays the poset of non-degenerate matrix classes in $\mathsf{Mclex}[4,5,2]$.

As we increase the parameters of $\mathsf{Mclex}[n,m,k]$, the posets get more and more complex. We have included here the one for $\mathsf{Mclex}[4,6,2]$ --- see Figure~\ref{fig:4x6x2}, which is the largest one with $n=4$ and $k=2$ that can fit with reasonable readability on a page. Figure~\ref{fig:sizes of 4xmx2} shows the growth of the sizes of posets $\mathsf{Mclex}[4,m,2]$ (including the degenerate matrix classes).

\begin{figure}
    \centering
	\includegraphics[width=455pt]{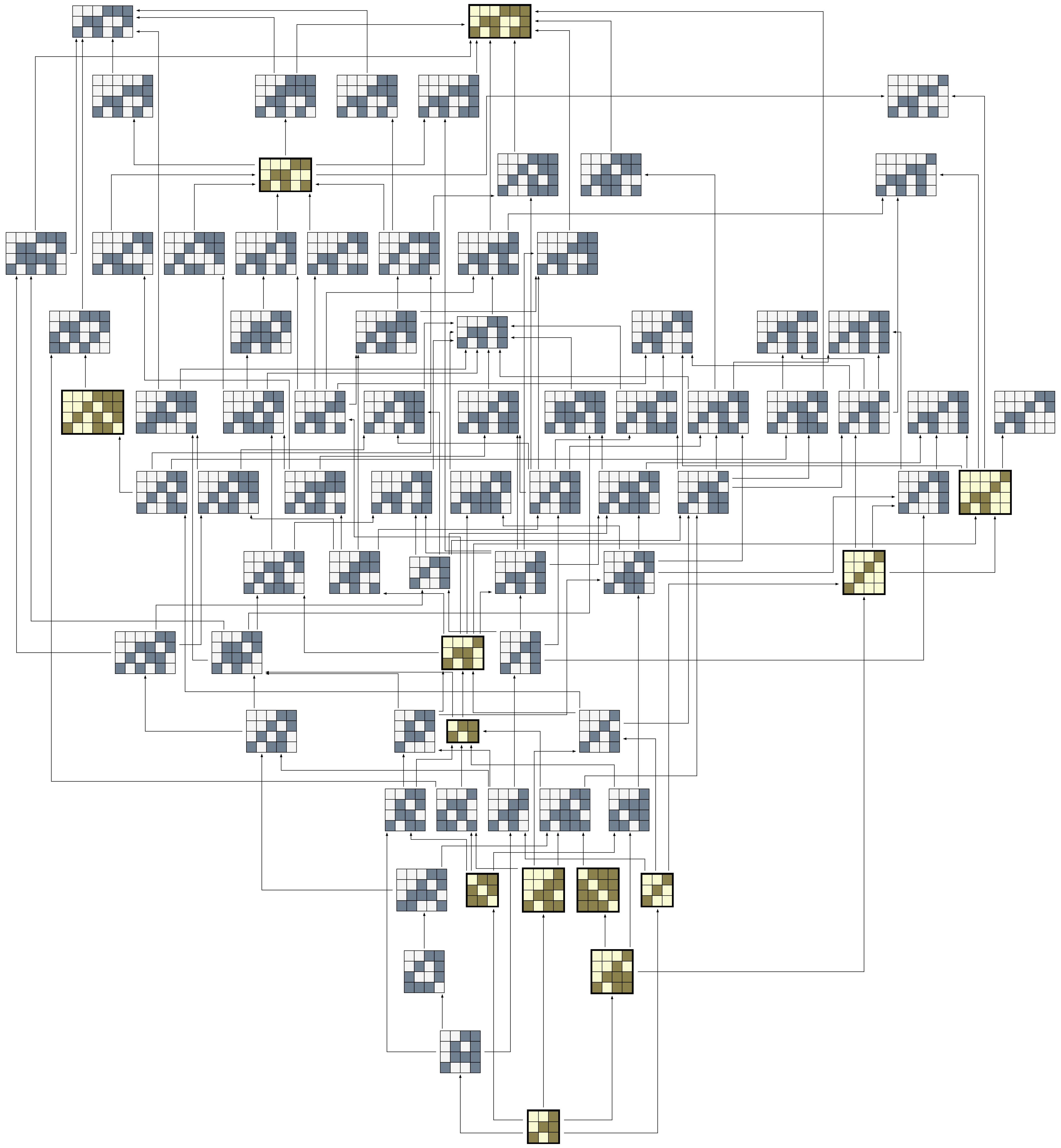}
    \caption{Hasse diagram of the poset of non-degenerate matrix classes in $\mathsf{Mclex}[4,6,2]$. The highlighted matrices with thicker frames are some of those that link with known Mal'tsev conditions. In vertical order from bottom to top and horizontally from right to left, these matrices arise from Mal'tsev conditions described by the following keywords:
    (1)~Arithmetical ($=$~Mal'tsev with majority);
    (2)~Mal'tsev with $4$-ary near unanimity;
    (3)~Majority;
    (4)~$4$-minority;
    (5)~Refinement of Mal'tsev with directly decomposable congruences;
    (6)~Minority;
    (7)~Mal'tsev;
    (8)~$3$-edge;
    (9)~$4$-ary near unanimity;
    (10)~$4$-edge;
    (11)~Refinement of directly decomposable congruence classes ($=$~refinement of local anti-commutativity);
    (12)~Refinement of the egg-box property;
    (13)~Refinement of normal local projections.}
    \label{fig:4x6x2}
\end{figure}

\begin{figure}[htb]
	\centering
	
	\begin{tikzpicture}[yscale=0.6, xscale=0.6]
	\begin{axis}[xlabel={$m$},ylabel={$|\mathsf{Mclex}[4,m,2]|$}]
	\addplot[color=black,mark=*] coordinates {
		(1,2)
		(2,2)
		(3,6)
		(4,21)
		(5,37)
		(6,79)
		(7,150)
		(8,241)
		(9,329)
		(10,384)
		(11,414)
		(12,430)
		(13,437)
		(14,440)
		(15,441)
		(16,441)
	};
	\end{axis}
	\end{tikzpicture}
	$\quad\quad$
	\begin{tikzpicture}[yscale=0.7, xscale=0.7]
    \begin{axis}[xlabel={$m$},ylabel={$|\mathsf{Mclex}[4,m,2]|-|\mathsf{Mclex}[4,m-1,2]|$}]
	\addplot[color=black,mark=*] coordinates {
		(2,0)
		(3,4)
		(4,15)
		(5,16)
		(6,42)
		(7,71)
		(8,91)
		(9,88)
		(10,55)
		(11,30)
		(12,16)
		(13,7)
		(14,3)
		(15,1)
		(16,0)
	};
	\end{axis}
    \end{tikzpicture}
	
	$\begin{array}{|r|c|c|c|c|c|c|c|c|c|c|c|c|c|c|c|c|}
	\hline
	m= & 1 & 2 & 3 & 4 & 5 & 6 & 7 & 8 & 9 & 10 & 11 & 12 & 13 & 14 & 15+\\
	\hline
	 \textrm{count} & 2 & 2 & 6 &  21 &  37 &  79 &  150 &  241 &  329 &  384 &  414 &  430 &  437 &  440 &  441\\\hline
	\end{array}$
	\caption{Count of matrix classes in $\mathsf{Mclex}[4,m,2]$ (including the two degenerate matrix classes that can be obtained from non-empty matrices).}
	\label{fig:sizes of 4xmx2}
\end{figure}
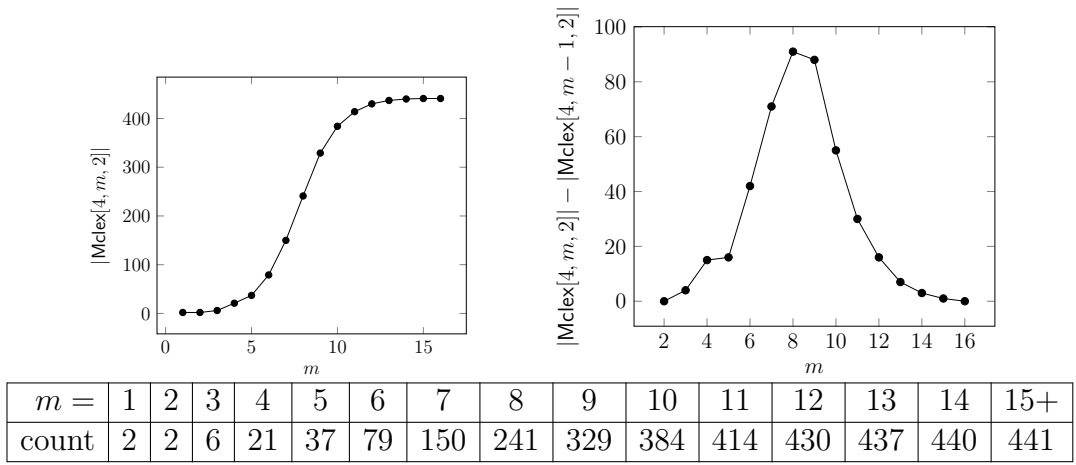

Figure~\ref{fig:4x6x2} shows some of the matrix classes that arise from matrices that are obtained from known Mal'tsev conditions in the literature or their syntactical refinements in the sense of Remark~\ref{RemB} (although we have not found explicitly the term `4-ary minority' in the literature). We will discuss some of these examples below.

Syntactical refinement of the Mal'tsev condition given in Theorem~1 of~\cite{Chajda1988} for the egg-box property yields a matrix that defines the matrix class appearing in the left-most place of the third row (from the top) of Figure~\ref{fig:4x6x2}. This matrix class appears already in Figure~\ref{fig:3xmx2}: right matrix in the third row.

The refinement of the Mal'tsev condition characterising varieties having `normal local projections' in the sense of~\cite{Janelidze2004} is given by:
$$\left\{\begin{array}{c} p(x_1,x_1,x_0,x_0,x_0,x_1)=x_0,\\ p(x_0,x_0,x_1,x_1,x_0,x_1)=x_0,\\
p(x_1,x_0,x_1,x_0,x_1,x_1)=x_0.\\ 
\end{array}\right.$$
The corresponding matrix, doubly lexi-ordered, is 
$$\left[\begin{array}{cccccc} 
0 & 0 & 0 & 1 & 1 & 1\\
0 & 1 & 1 & 0 & 0 & 1\\
1 & 0 & 1 & 0 & 1 & 1
\end{array}\right].$$
It coincidentally matches with the right matrix in the top row of Figure~\ref{fig:4x6x2}, which is the same as the left matrix in the second row of Figure~\ref{fig:3xmx2}.

Although Figure~\ref{fig:4x6x2} already has a large number of representations of matrix classes, there are some arising in the literature that cannot be found there. For instance, using the computer it is possible to establish that the syntactical refinement of the Mal'tsev condition defining varieties with `difunctional class relations' from~\cite{HoefnagelJanelidzeRodelo2020} gives a matrix class which does not live in $\mathsf{Mclex}[4,6,2]$, and so it cannot be found in Figure~\ref{fig:4x6x2}. This matrix class is the top one in Figure~\ref{fig:difclasrel}.

\begin{figure}
    \centering
	\includegraphics[width=350pt]{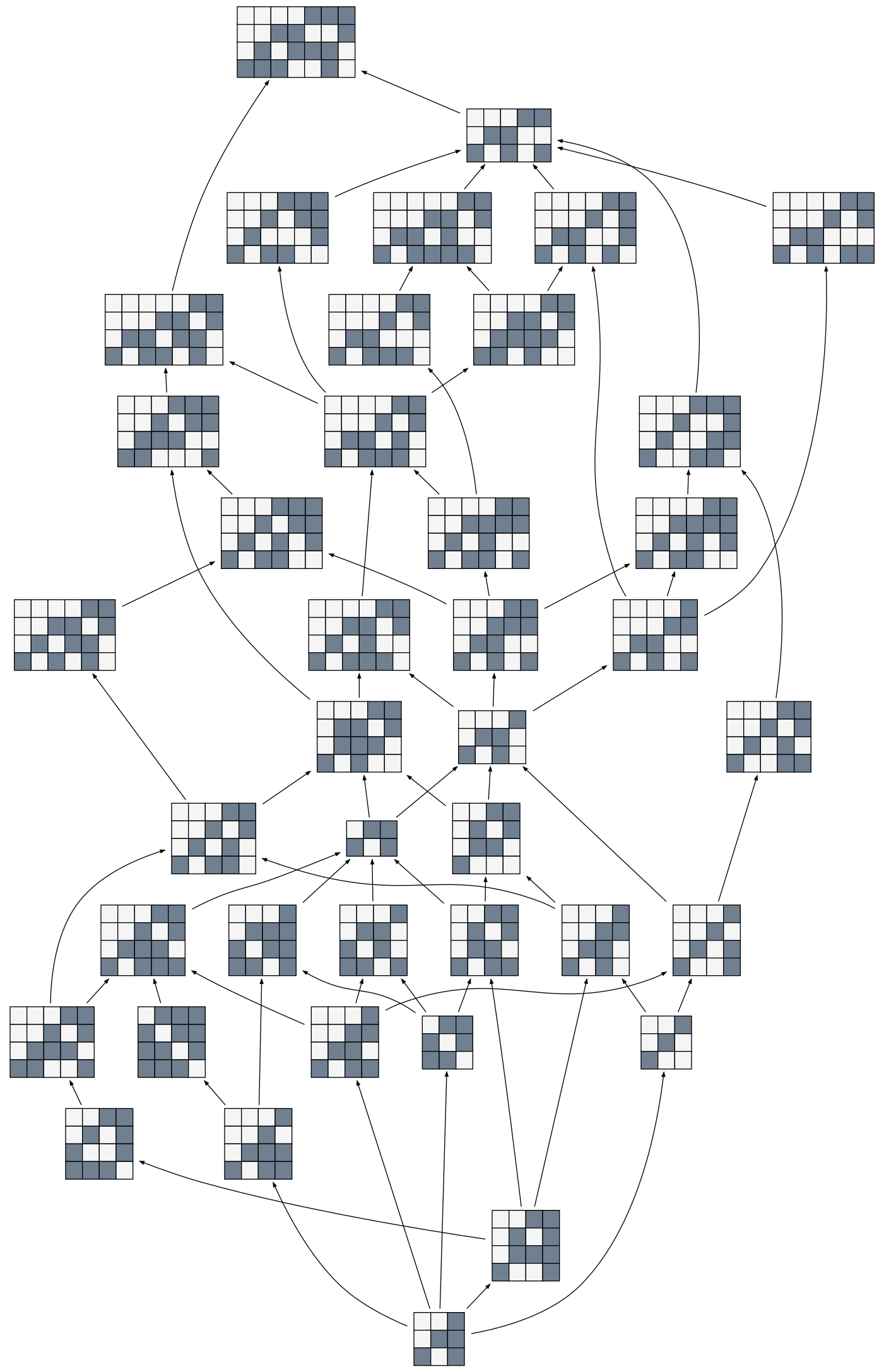}
    \caption{Hasse diagram of the poset of non-degenerate matrix classes in $\mathsf{Mclex}[4,7,2]$ that are contained in the matrix class given by the syntactical refinement of the Mal'tsev condition for difunctionality of class relations from~\cite{HoefnagelJanelidzeRodelo2020}.}
    \label{fig:difclasrel}
\end{figure}

We have also made some attempt to look at the cases when $k\in\{3,4\}$. For instance, we were able to generate Figure~\ref{fig:3x5x4} representing the poset of non-degenerate matrix classes in $\mathsf{Mclex}[3,5,4]$. In that picture, the darker a square is, the higher the value of the corresponding entry is. We have included there the digits in each square to improve readability. We also have been able to compute the posets $\mathsf{Mclex}[3,m,3]$ for $m\leqslant 9$ --- see Figure~\ref{fig:sizes of 3xmx3} for the number of matrix classes in these posets.

\begin{figure}
    \centering
	\includegraphics[width=273pt]{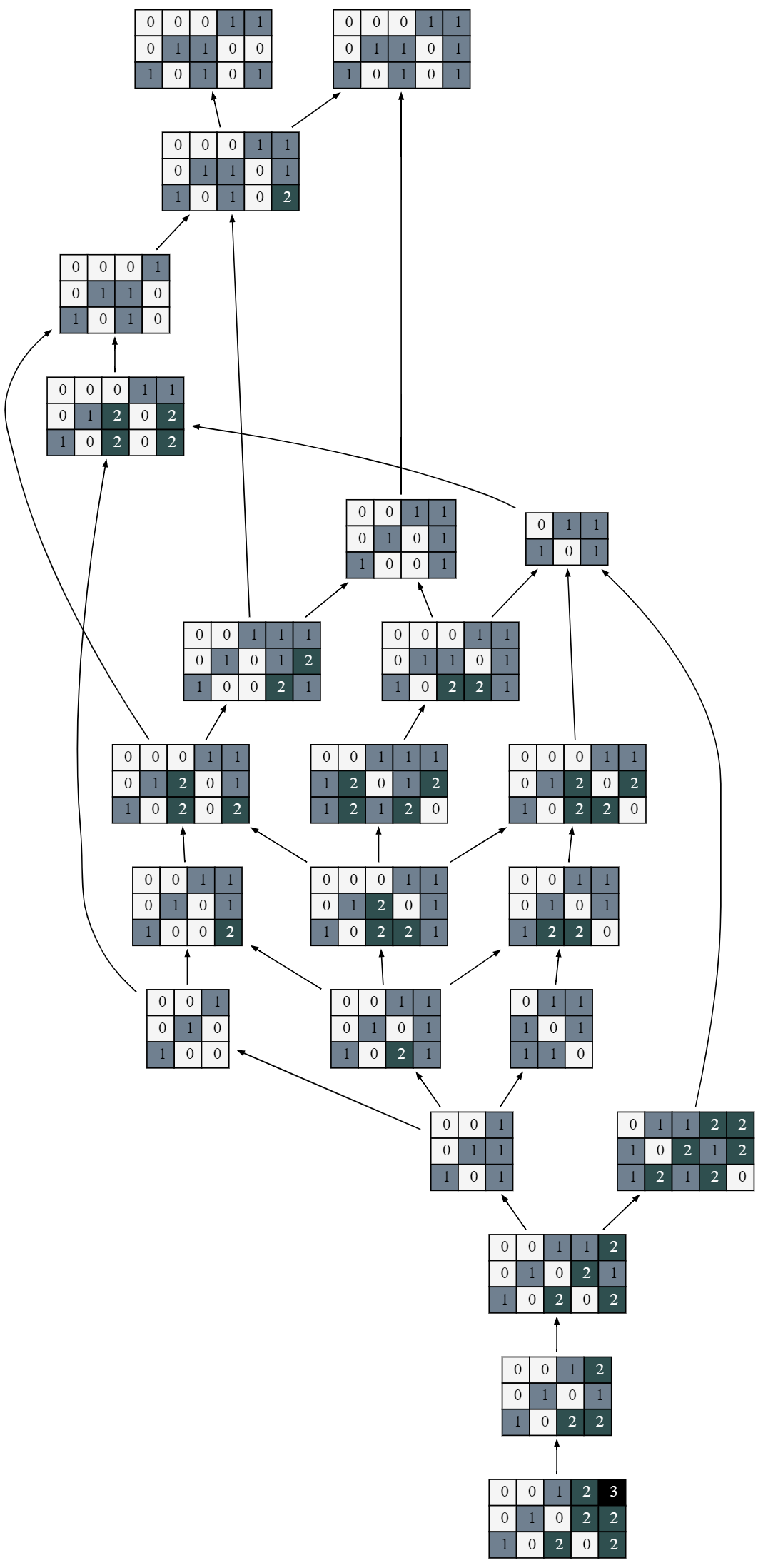}
    \caption{Hasse diagram of the poset of non-degenerate matrix classes in $\mathsf{Mclex}[3,5,4]$. Note that the bottom matrix class is the only one in the figure which does not belong to $\mathsf{Mclex}[3,5,3]$.}
    \label{fig:3x5x4}
\end{figure}

\begin{figure}[htb]
	\centering
	
	\begin{tikzpicture}[yscale=0.6, xscale=0.6]
	\begin{axis}[xlabel={$m$},ylabel={$|\mathsf{Mclex}[3,m,3]|$}]
	\addplot[color=black,mark=*] coordinates {
		(1,2)
		(2,2)
		(3,6)
		(4,12)
		(5,24)
		(6,78)
		(7,233)
		(8,812)
		(9,2384)
	};
	\end{axis}
	\end{tikzpicture}
	$\quad\quad$
	\begin{tikzpicture}[yscale=0.7, xscale=0.7]
    \begin{axis}[xlabel={$m$},ylabel={$|\mathsf{Mclex}[3,m,3]|-|\mathsf{Mclex}[3,m-1,3]|$}]
	\addplot[color=black,mark=*] coordinates {
		(2,0)
		(3,4)
		(4,6)
		(5,12)
		(6,54)
		(7,155)
		(8,579)
		(9,1572)
	};
	\end{axis}
    \end{tikzpicture}
	
	$\begin{array}{|r|c|c|c|c|c|c|c|c|c|c|c|c|c|c|c|c|}
	\hline
	m= & 1 & 2 & 3 & 4 & 5 & 6 & 7 & 8 & 9\\
	\hline
	 \textrm{count} & 2 &  2 & 6 &  12 &  24 &  78 &  233 & 812 & 2384\\\hline
	\end{array}$
	\caption{Count of matrix classes in $\mathsf{Mclex}[3,m,3]$ for $m\leqslant 9$ (including the two degenerate matrix classes that can be obtained from non-empty matrices).}
	\label{fig:sizes of 3xmx3}
\end{figure}

The matrix considered in the beginning of the Introduction (which was, by the way, chosen randomly), turns out to determine the matrix class that appears farthest to the right in the eighth row of Figure~\ref{fig:3x5x4}. This can be seen from the following two lex-tableaux, proposed by the computer:
$$\begin{array}{ccccc|cccc}
1 & 1 & 0 & 2 & 2 & 0 & 0 & 1 & 1\\
0 & 0 & 1 & 1 & 0 & 0 & 1 & 0 & 1\\
2 & 0 & 1 & 2 & 1 & 1 & 2 & 2 & 0\\
\hline
0 & 0 & 1 & 0 & 1 & 0\\
0 & 0 & 1 & 1 & 0 & 0\\
1 & 1 & 0 & 2 & 2 & 0\\
\end{array}
\quad\quad\quad
\begin{array}{cccc|ccccc}
0 & 0 & 1 & 1 & 1 & 1 & 0 & 2 & 2\\
0 & 1 & 0 & 1 & 0 & 0 & 1 & 1 & 0\\
1 & 2 & 2 & 0 & 2 & 0 & 1 & 2 & 1\\
\hline
2 & 2 & 1 & 1 & 2\\
1 & 1 & 0 & 0 & 1\\
2 & 2 & 2 & 0 & 0\\
\hline
1 & 2 & 2 & 0 & 0\\
0 & 0 & 1 & 1 & 0\\
0 & 1 & 0 & 1 & 0\\
\end{array}$$

Given positive integers $n_1,n_2,m_1,m_2,k_1,k_2$, we obviously have the poset inclusion
$$\mathsf{Mclex}[\mathsf{min}(n_1,n_2),\mathsf{min}(m_1,m_2),\mathsf{min}(k_1,k_2)] \subseteq \mathsf{Mclex}[n_1,m_1,k_1] \cap \mathsf{Mclex}[n_2,m_2,k_2].$$
The question of the converse inclusion is in general not easy. Using the computer, we found the following equalities:
$$\mathsf{Mclex}[3,7,2] \cap \mathsf{Mclex}[3,5,4] = \mathsf{Mclex}[3,5,2],$$
$$\mathsf{Mclex}[3,7,2] \cap \mathsf{Mclex}[4,5,2] = \mathsf{Mclex}[3,5,2]$$
and
$$\mathsf{Mclex}[3,7,2] \cap \mathsf{Mclex}[4,6,2] = \mathsf{Mclex}[3,6,2].$$
However, in other cases, the inclusion may be strict. For instance, the computer calculated that the intersection $\mathsf{Mclex}[3,9,3] \cap \mathsf{Mclex}[4,15,2]$ contains exactly six matrix classes that do not belong to $\mathsf{Mclex}[3,9,2]=\mathsf{Mclex}[3,7,2]$. One of these matrix classes is given by the matrices
$$M_1 = \left[\begin{array}{ccccccc} 
0 & 0 & 0 & 1 & 1\\
0 & 1 & 1 & 0 & 1\\
1 & 0 & 2 & 2 & 1
\end{array}\right]$$
which is $(3,9,3)$-canonical (and thus also $(3,5,3)$-canonical) and
$$M'_1 = \left[\begin{array}{ccccc} 
0 & 0 & 0 & 1 & 1\\
0 & 0 & 1 & 0 & 1\\
0 & 1 & 1 & 1 & 0\\
1 & 1 & 0 & 0 & 1
\end{array}\right]$$
which is $(4,15,2)$-canonical (and thus also $(4,5,2)$-canonical) and such that $\mathsf{mclex}\{M_1\}=\mathsf{mclex}\{M'_1\}$. Therefore, the matrix $M'_1$ cannot be $(4,5,3)$-canonical. This matrix class appears in the right place of the sixth row of Figure~\ref{fig:3x5x4} represented by $M_1$ but appears in the fourth place of the seventh row of Figure~\ref{fig:4x5x2} represented by $M'_1$ (and of course also in Figures~\ref{fig:4x6x2} and~\ref{fig:difclasrel}). This is the only matrix class that has been represented by two different matrices in our Figures~\ref{fig:4x5x2}--\ref{fig:3x5x4}. The five other elements of the intersection $\mathsf{Mclex}[3,9,3] \cap \mathsf{Mclex}[4,15,2]$ which are not in $\mathsf{Mclex}[3,7,2]$ are given by $\mathsf{mclex}\{M_i\}=\mathsf{mclex}\{M'_i\}$, for $i\in\{2,3,4,5,6\}$, where $M_i$ is a $(3,9,3)$-canonical matrix and $M'_i$ is a $(4,15,2)$-canonical matrix given by
\begin{align*}
M_2 = \left[\begin{array}{ccccccc} 
0 & 0 & 0 & 0 & 1 & 1 & 1\\
0 & 1 & 1 & 2 & 0 & 0 & 2\\
1 & 0 & 2 & 2 & 0 & 2 & 1
\end{array}\right]\quad &\text{ and } \quad
M'_2 = \left[\begin{array}{ccccc} 
0 & 0 & 0 & 1 & 1\\
0 & 0 & 1 & 0 & 1\\
0 & 1 & 0 & 1 & 0\\
1 & 0 & 1 & 1 & 0
\end{array}\right],\\
M_3 = \left[\begin{array}{cccccccc} 
0 & 0 & 0 & 0 & 0 & 1 & 1 & 1\\
0 & 1 & 1 & 1 & 2 & 0 & 1 & 2\\
1 & 0 & 1 & 2 & 2 & 0 & 2 & 0
\end{array}\right]\quad &\text{ and } \quad
M'_3 = \left[\begin{array}{cccccc} 
0 & 0 & 0 & 0 & 1 & 1\\
0 & 0 & 0 & 1 & 0 & 1\\
0 & 1 & 1 & 0 & 1 & 0\\
1 & 0 & 1 & 1 & 1 & 0
\end{array}\right],\\
M_4 = \left[\begin{array}{cccccccc} 
0 & 0 & 0 & 0 & 1 & 1 & 1 & 1\\
0 & 1 & 1 & 1 & 0 & 0 & 0 & 1\\
1 & 0 & 1 & 2 & 0 & 1 & 2 & 2
\end{array}\right]\quad &\text{ and } \quad
M'_4 = \left[\begin{array}{cccccccc} 
0 & 0 & 0 & 0 & 0 & 1 & 1 & 1\\
0 & 0 & 0 & 1 & 1 & 0 & 0 & 1\\
0 & 1 & 1 & 0 & 1 & 0 & 1 & 0\\
1 & 0 & 1 & 1 & 0 & 1 & 1 & 0
\end{array}\right],\\
M_5 = \left[\begin{array}{ccccccccc} 
0 & 0 & 0 & 0 & 0 & 1 & 1 & 1 & 1\\
0 & 1 & 1 & 1 & 2 & 0 & 1 & 2 & 2\\
1 & 0 & 1 & 2 & 1 & 0 & 2 & 1 & 2
\end{array}\right]\quad &\text{ and } \quad
M'_5 = \left[\begin{array}{ccccccc} 
0 & 0 & 0 & 0 & 0 & 1 & 1\\
0 & 0 & 0 & 1 & 1 & 0 & 1\\
0 & 1 & 1 & 0 & 1 & 1 & 0\\
1 & 0 & 1 & 1 & 0 & 1 & 0
\end{array}\right],\\
M_6 = \left[\begin{array}{ccccccccc} 
0 & 0 & 0 & 0 & 1 & 1 & 2 & 2 & 2\\
0 & 1 & 1 & 2 & 0 & 2 & 1 & 2 & 2\\
1 & 0 & 2 & 2 & 0 & 1 & 0 & 0 & 2
\end{array}\right]\quad &\text{ and } \quad
M'_6 = \left[\begin{array}{cccccc} 
0 & 0 & 0 & 0 & 1 & 1\\
0 & 0 & 1 & 1 & 0 & 1\\
0 & 1 & 0 & 1 & 1 & 0\\
1 & 0 & 1 & 0 & 1 & 0
\end{array}\right].
\end{align*}

\section{Context sensitivity}\label{sec:context sensitivity}

Let us now address the following question. Restricted only to algebraic categories (i.e., categories of algebras in a variety), our matrix properties are a particular type of Mal'tsev conditions, and so it is possible that implication of these Mal'tsev conditions is equivalent to implication of the corresponding matrix properties. We will now show that this is \emph{not} the case. In other words, our poset $\mathsf{Mclex}$ is not `visible' at the level of collections of varieties of universal algebras. Note that replacing varieties with quasi-varieties would not help either, since, as it can be established, a quasi-variety belongs to a given matrix class if and only if the generated variety does. Moreover, the counter-example described below also indicates that there are implications of matrix properties that hold for regular well-powered categories, but not for all finitely complete categories.

Algebraic categories having $D_n$-closed relations, where $D_n$ is the same as at the start of Section~\ref{sec4}, are given by varieties having the so-called `$n$-ary near unanimity term'. As shown in~\cite{Mitchke1978}, for each given $n>2$, such a variety is `congruence distributive' (i.e., congruence lattices are distributive). As remarked in~\cite{Pixley1979}, this means that algebraic Mal'tsev categories with $D_n$-closed relations are nothing but arithmetical varieties. As proved in~\cite{HoefnagelPhD}, this result extends to Barr-exact categories, and thus, according to~\cite{Jacqmin2020a}, to all regular well-powered categories: given $n>2$, a regular well-powered category has $M$-closed relations and $D_n$-closed relations if and only if it has $A$-closed relations, where $M$ is the matrix that determines the matrix class of Mal'tsev categories (Example~\ref{ExaA}) and $A$ is the matrix for arithmetical varieties (Example~\ref{ExaC}). However, we know from~\cite{Hoefnagel2019a} that this result does not extend to finitely complete categories, i.e., the inclusion in
$$\mathsf{mclex}\{M\}\cap\mathsf{mclex}\{D_3\}=\mathsf{mclex}\{A\}\subsetneqq \mathsf{mclex}\{M\}\cap\mathsf{mclex}\{D_4\}$$
is strict. We can (re-)prove the above equality and inclusion using lex-tableaux and the computer-aided classification results. The matrix $M$ is a reduction of~$A$, so the matrix class for $A$ is contained in the matrix class for~$M$. That the matrix class for $A$ is contained in the matrix class for $D_3$ can be established by the following lex-tableau (where we used doubly lexi-ordered versions of our matrices):
$$\begin{array}{ccc|ccc}
0 & 0 & 1 & 0 & 0 & 1\\
0 & 1 & 1 & 0 & 1 & 0 \\
1 & 0 & 1 & 1 & 0 & 0\\
\hline
0 & 0 & 1 & 0\\
1 & 0 & 0 & 1\\
0 & 1 & 0 & 1\\
\hline
0 & 0 & 0 & 0\\
1 & 0 & 1 & 0\\
0 & 1 & 1 & 0
\end{array}$$
This proves $$\mathsf{mclex}\{A\}\subseteq\mathsf{mclex}\{M\}\cap\mathsf{mclex}\{D_3\}.$$
The converse inclusion can be established by the following lex-tableau; we hope the reader will be able to understand the slightly different form of the tableau resulting from the fact that we now have two matrices in the premise:
$$\begin{array}{ccc:ccc|ccc}
& & & 0 & 0 & 1 & 0 & 0 & 1\\
0 & 1 & 1 & 0 & 1 & 0 & 0 & 1 & 1 \\
1 & 0 & 1 & 1 & 0 & 0 & 1 & 0 & 1 \\
\hline
& & & 0 & 0 & 1 & 0\\
& & & 1 & 0 & 1 & 1\\
& & & 0 & 1 & 1 & 1\\
\hline
0 & 0 & 0 &&&& 0\\
0 & 1 & 1 &&&& 0\\
1 & 0 & 1 &&&& 0\\
\end{array}$$
This proves the desired equality. The strict inclusion $\mathsf{mclex}\{A\}\subsetneqq \mathsf{mclex}\{M\}\cap\mathsf{mclex}\{D_4\}$ follows from Figure~\ref{fig:4x4x2}, since as we will now show, $\mathsf{mclex}\{M\}\cap\mathsf{mclex}\{D_4\}$ is given by the left matrix in the fourth row of Figure~\ref{fig:4x4x2}, which we will denote by~$A'$. Figure~\ref{fig:4x4x2} shows that $$\mathsf{mclex}\{A'\}\subseteq \mathsf{mclex}\{M\}\cap\mathsf{mclex}\{D_4\}.$$ The converse is established by the following lex-tableau:
$$\begin{array}{ccc:cccc|cccc}
& & & 0 & 0 & 0 & 1 & 0 & 0 & 0 & 1  \\
& & & 0 & 0 & 1 & 0 & 0 & 0 & 1 & 0  \\
0 & 1 & 1 & 0 & 1 & 0 & 0 & 0 & 1 & 1 & 1  \\
1 & 0 & 1 & 1 & 0 & 0 & 0 & 1 & 0 & 1 & 1  \\
\hline
& & & 0 & 0 & 0 & 1 & 0\\
& & & 0 & 0 & 1 & 0 & 0\\
& & & 1 & 0 & 1 & 1 & 1\\
& & & 0 & 1 & 1 & 1 & 1\\
\hline
0 & 0 & 0 &&&&& 0\\
0 & 0 & 0 &&&&& 0\\
0 & 1 & 1 &&&&& 0\\
1 & 0 & 1 &&&&& 0
\end{array}$$
This proves $\mathsf{mclex}\{A'\}=\mathsf{mclex}\{M\}\cap\mathsf{mclex}\{D_4\}.$ That $\mathsf{mclex}\{A'\}\nsubseteq \mathsf{mclex}\{A\}$ also has a simple counter-example given by $(\mathsf{Inj}_{\{M\}}\mathbf{Rel}_3)^\mathsf{op}$, i.e., the dual of the category of all ternary relations that are $M$-sharp (where $M$ is the Mal'tsev matrix). This category is not a majority category, yet any ternary relation is $D_4$-sharp, since any $3$-row reduction of $D_4$ contains a left column that is identical to the right column (see Proposition~3.25 in~\cite{Hoefnagel2019a}).

As an another counter-example, let us consider the matrix class on the right of the fourth row of Figure~\ref{fig:3xmx2}. Algebraic categories belonging to that matrix class are given by the varieties which admit a `$3$-edge term' in the sense of~\cite{BermanIdziakMarkovicMcKenzieValerioteWillard2010}. As it is shown in the mentioned paper, a variety admits a $3$-edge term if and only if it admits a `$3$-cube term', i.e., if and only if it belongs to the top matrix class of Figure~\ref{fig:3xmx2}. This means that the matrix on the right of the fourth row of Figure~\ref{fig:3xmx2} and the five other matrices above it in that figure all give equivalent conditions on algebraic categories. In a similar way, this gives that $37$ of the matrix classes in Figure~\ref{fig:4x6x2} coincide when restricted to algebraic categories. We can also deduce from this result that $227$ matrix classes in $\mathsf{Mclex}[4,15,2]$ coincide when restricted to algebraic categories. Moreover, as it shown in~\cite{BermanIdziakMarkovicMcKenzieValerioteWillard2010}, a variety admits a `$4$-edge term' if and only if it admits a `$4$-cube term'. Using this result, we can deduce that $173$ other matrix classes in $\mathsf{Mclex}[4,15,2]$ coincide when restricted to algebraic categories. Combining these two facts, we know that the $441$ different matrix classes in $\mathsf{Mclex}[4,15,2]$ give at most $43$ non-equivalent Mal'tsev conditions on varieties.

So, implications of matrix properties are `context sensitive': their validity depends on what further exactness properties the base category has, such as being finitely complete or algebraic. We remark here that, as shown in~\cite{Jacqmin2020a}, matrix properties are stable under the `exact completion' $\C \hookrightarrow \C_{\mathsf{reg/ex}}$ of a regular well-powered category~\cite{Lawvere1972}. In particular, since the embedding $\C \hookrightarrow \C_{\mathsf{reg/ex}}$ is a fully faithful regular functor (see e.g.~\cite{CarboniVitale1998}), it follows that implications of matrix properties are not sensitive to whether the base regular category is Barr-exact or not (assuming the axiom of universes~\cite{ArtinGrothendieckVerdier1972} to avoid size issues). In general, it would be interesting to investigate implications of matrix properties, in a similar way as we have done here, in the contexts of algebraic categories and of regular categories, and to analyse their differences. One could try to do the algebraic case using the characterisation~\cite{ZJanelidze2006a} of matrix properties in that context as Mal'tsev conditions. As for the regular case, the technique of the present paper cannot be directly applied there since the categories $(\mathsf{Inj}_S\mathbf{Rel}_n)^\mathsf{op}$ are in general not regular. However, using the results from~\cite{Jacqmin2018,Jacqmin2020a,JacqminJanelidze2021} (and again assuming the axiom of universes), proving an implication of matrix properties for all regular categories is equivalent to proving it only for regular locally presentable categories~\cite{GabrielUlmer1971}, or equivalently, for regular essentially algebraic categories~\cite{AdamekHerrlichRosicky1988,AdamekRosicky1994}. Therefore, the techniques one could use to solve the algebraic case might have a counterpart in the essentially algebraic world to solve the regular case. In view of Theorem~\ref{ThmF}, since the categories $\mathsf{Inj}_S\mathbf{Rel}_n$ are essentially algebraic, the situation is somehow dual in the finitely complete context: proving an implication of matrix properties for all finitely complete categories is equivalent to proving it only for the duals of essentially algebraic categories.

\begin{remark}
We should point out a mistake in the statement/proof of Proposition~3.23 in~\cite{Hoefnagel2019a} (which is also Proposition~2.36 in~\cite{HoefnagelPhD}, and which is referenced in the concluding remarks section of~\cite{Hoefnagel2020a}). That proposition claims that, for the matrix $M$ that determines the matrix class of Mal'tsev categories (Example~\ref{ExaA}), the category $(\mathsf{Inj}_{\{M\}}\mathbf{Rel}_3)^\mathsf{op}$ is a regular category --- however, $(\mathsf{Inj}_{\{M\}}\mathbf{Rel}_3)^\mathsf{op}$ is not a regular category.
\end{remark}

Let us give another example showing context-sensitivity of implications of matrix properties. The Mal'tsev condition for `local-anticommutativity' given in~\cite{Hoefnagel2020c} is the Mal'tsev condition for direct decomposability of congruence classes~\cite{Duda1986} together with some identities linking some of the `inner' terms involved in the condition. This results in the syntactical refinement of the two Mal'tsev conditions being the same. The corresponding matrix class is the left one in the sixth row (from the top) of Figure~\ref{fig:4x6x2} which, as we can see on that figure, contains the third one in the twelfth row. As it follows from Remark~2.25 in~\cite{Hoefnagel2020c}, a Mal'tsev variety is locally anticommutative if and only if it is arithmetical. This implies that, in the algebraic case, the matrix class appearing in the third place of the twelfth row of Figure~\ref{fig:4x6x2} (which is the same as the matrix class from Remark~\ref{RemB}) matches with the matrix class given by the arithmetical matrix (Example~\ref{ExaB}). In the finitely complete context, this is not the case as shown by Figure~\ref{fig:4x6x2}.

We conclude by mentioning that, however implications of matrix properties are context sensitive in general, it is shown in~\cite{HoefnagelJacqmin2022} that the inclusion $\mathsf{mclex}\{N\}\subseteq \mathsf{mclex}\{M\}$, where $M$ is the Mal'tsev matrix and $N\in\M(n,m,k)$ is any matrix, is equivalent to the analogue inclusion in the algebraic context. Moreover, it is shown in~\cite{HoefnagelJacqminJanelidzeVanderWalt2022} that, if $N$ only contains $0$'s and $1$'s (i.e. $k=2$), this implication is further equivalent to $\mathsf{mclex}\{M'\}\nsubseteq \mathsf{mclex}\{N\}$ where $M'$ is the majority matrix of Example~\ref{ExaB}.

\end{document}